\documentclass{article}
\usepackage{amssymb}
 \usepackage{amsthm}
\usepackage{amsmath,amssymb,amsopn,amsfonts,mathrsfs,amsbsy,amscd}
\usepackage{longtable}
\usepackage{multirow}
\usepackage[latin1]{inputenc}
\newcommand{\mL}{\mathrm{L}}
\newcommand{\aff}{\mathfrak{aff}(2,\R)}
\newcommand{\msl}{\mathfrak{sl}_2(\R)}
\newcommand{\mso}{\mathfrak{so}(3)}
\newcommand{\ad}{\mathrm{ad}}
\newcommand{\tr}{\mathrm{tr}}
\newcommand{\lr}{\longrightarrow}
\newcommand{\esp}{\quad\mbox{and}\quad}
\newcommand{\R}{\mathbb{R}}

\newcommand{\ind}{\mathrm{ind}}
\newcommand{\M}{\mathrm{M}}
\newcommand{\rank}{\mathrm{rank}}
\newcommand{\al}{\alpha}
\newcommand{\G}{{\mathfrak{g}}}

\newcommand{\h}{{\mathfrak{h}}}
\newcommand{\J}{{\mathrm{J}}}

\newcommand{\so}{\mathfrak{so}}
\newcommand{\om}{\omega}

\newcommand{\e}{\check{e}}

\newtheorem{theorem}{Theorem}
\newtheorem{proposition}{Proposition}
\newtheorem{lemma}{Lemma}
\newtheorem{corollary}{Corollary}

\newtheorem{remark}{Remark}

\begin{document}
\begin{center}
\textbf{Eight-dimensional symplectic  non-solvable  Lie algebras}
\end{center}
\begin{center}
T. A\"it Aissa and M. W. Mansouri
\end{center}
\begin{center}
Department of Mathematics, Faculty of Sciences, Ibn Tofail University\\
	Analysis, Geometry and Applications Laboratory (LAGA)\\ Kenitra, Morocco\\e-mail:
	mansourimohammed.wadia@uit.ac.ma\\tarik.aitaissa @uit.ac.ma
\end{center}
\rule{1\textwidth}{1pt}
\begin{flushleft}
\textbf{Abstract}
\end{flushleft}
\begin{flushleft}
In this paper, we classify eight-dimensional  non-solvable  Lie algebras that support a symplectic structure. As well as a complete
classification is given, up to symplectomorphism, of eight-dimensional symplectic non-solvable  Lie algebras.
\end{flushleft}
\begin{flushleft}
        MSC  17B30, \\ MSC 17B60, \\ MSC 53D05.
\end{flushleft}
       \rule{1\textwidth}{1pt}

\section{Introduction and main result}
Let $\G$ be a finite-dimensional real Lie algebra. We say that
$(\G,\om)$ is a {\it{symplectic Lie algebra}} if $\om$ is a
non-degenerate skew-symmetric bilinear form on $\G$ and
\begin{equation}\label{cocy}
\big(\mathrm{d}\om\big)(x,y,z):=\om([x, y], z)+\om([y, z], x)+\om([z, x],y) = 0,
\end{equation}
for all $x$, $y$, $z\in\G$, where $\mathrm{d}$ is the Chevalley-Eilenberg differential. This is to say, $\om$ is a non-degenerate
$2$-cocycle for the scalar cohomology of $\G$. Note that in such case,
$\G$ must be even-dimensional. We will then call $\om$ a symplectic
structure on $\G$. A fundamental class of symplectic Lie algebras is
 formed by  Frobenius Lie algebras, i.e., Lie algebras
admitting a non-degenerate exact 2-form. Symplectic Lie algebras are
in one-to-one correspondence with simply connected Lie groups with
left invariant symplectic forms.

Recall that two symplectic Lie algebras $(\G_1,\om_1)$ and $(\G_2,\om_2)$ are said to be
symplectomorphically equivalent if there exists an isomorphism of Lie algebras $\varphi : \G_1\lr\G_2$, which preserves the symplectic forms, that is $\varphi^*\om_2=\om_1$.

The study of symplectic Lie algebras is an active area of research.
The characterization problem of symplectic Lie algebras is still an
open problem, even though there are many interesting results on
obstructions on a Lie algebra to support a symplectic structure. Let
us recall the following well-known results (see \cite{C}).
\begin{enumerate}
        \item A semisimple Lie algebra (in particular if $[\G, \G] = \G$) does
not admit symplectic structures.
        \item The direct sum of semisimple and solvable Lie algebras cannot
be symplectic.
        \item Unimodular symplectic Lie algebras are solvable.
        \item All symplectic Lie algebras of dimension four are solvable.
\end{enumerate}

From $1.$ and $2.$ a Lie algebra that supports a symplectic structure is
either solvable or admits a non-trivial Levi-Malcev decomposition
(i.e. $\G=N\ltimes R$ with Levi factor $N\not=0$, radical part
$R\not=0$ and $\ltimes\not=\oplus$). In what follows, this second
class will be called non-solvable.

To our knowledge, the classification of symplectic Lie algebras, up to
sympectomorphism, only exist for dimensions less than four (see
\cite{O}) and  six-dimensional   nilpotent symplectic Lie algebras (see
\cite{K-G-M} and  \cite{F} for a more recent list). A classification of
a large subfamily of six-dimensional non-nilpotent solvable Lie algebras
has been made by Stursberg (see \cite{S}). More precisely, it
describes symplectic structures on Lie algebras which decompose into a
direct sum of two ideals and indecomposable Lie algebras with a
four-dimensional nilradical. This classification covers all cases
except indecomposable Lie algebras with  the nilradical is five-dimension.  Nevertheless, there are particular results for any dimensions.
In this work, we  classify non-solvable symplectic Lie algebras in eight-dimensions.

A finite-dimensional algebra $(\G,\cdot)$  is called left-symmetric algebra (abbrev. LSA) if it
satisfies the identity
\begin{equation}\label{ass}
(x,y,z)=(y,x,z)\qquad \forall x,y,z \in\G,
\end{equation}
where $(x,y,z)$  denotes the associator $(x,y,z)=(x\cdot y
)\cdot z-x\cdot(y\cdot z)$.   In this case, the commutator $[x,y]= x\cdot y-y\cdot x$
defines a  bracket that makes $\G$  a Lie algebra.
 Let $\mathrm{L}(x)$ and $\mathrm{R}(x)$ denote the left and right multiplications by the element
$x\in\G$, respectively. The identity $(\ref{ass})$ is now equivalent to the formula
\begin{equation*}
[\mathrm{L}(x),\mathrm{L}(y)]=\mathrm{L}{([x,y])},\qquad \forall x,y\in \G,
\end{equation*}
or in other words, the linear map $\mathrm{L} : \G \longrightarrow \mathrm{End}(\G)$ is a representation of Lie algebras, but in general $\mathrm{R}$  is not an algebra homomorphism, an LSA is also noted $(\G,\mL)$. For more details on left-symmetric algebras, we refer the reader to the survey article \cite{B} and the references therein.

It is known that  (see \cite{C} and \cite{M-R}) that given a symplectic Lie algebra $(\G,\omega)$, the product given by
\begin{equation*}
\omega(x\cdot y,z)=-\omega(y,[x,z]),\qquad\forall x,y,z\in\G,
\end{equation*}
induces a left symmetric algebra structure $``\cdot$" on $\G$ that satisfies
$x\cdot y-y\cdot x=[x,y]$. If in addition, the symplectic Lie algebra
$(\G,\om)$ admits a Lagrangian ideal $\J$, then the quotient algebra
$\h=\G/\J$ admits a left symmetric product and  the  symplectic Lie
algebra $(\G,\om)$ can be reconstructed from the left symmetric
algebra $\h$ (see for now \cite{C}).

\textbf{Notation.}  For $\{e_i\}_{1\leq i\leq n}$  a basis of $\G$,
we denote by $\{e^i\}_{1\leq i\leq n}$ the
dual basis on $\G^\ast$ and  $e^{ij}$  the 2-form $e^i\wedge
e^j\in\wedge^2\G^*$. Set by  $\langle F \rangle:= \mathrm{span}_\R\{F\}$ the Lie
subalgebra  generated by the family $F$. The notation and indices for the Lie algebras correspond to those given by Turkowski in \cite{T}.

The main purpose of this article is to show the following theorem.
\begin{theorem}\label{the1}
Let $\G$ be an eight-dimensional symplectic  non-solvable  Lie algebra. Then
$\G$ is symplectomorphically equivalent to one of the following Lie algebras equipped
with a symplectic form as follows:
	{\renewcommand*{\arraystretch}{1.3}
	\begin{longtable}{ll}
			\hline
		Lie algebra& Symplectic structure\\
			\hline
	$\mL_{8,3}=\mso\ltimes A_{5,7}^{1,1,1}$& $\omega= e^{17}+e^{25}+e^{36}-2e^{48}$\\
	\hline
\multirow{2}{*}{$\mL^{p\not=0}_{8,4}=\mso\ltimes
A_{5,17}^{1,p,p}$	}&$\omega= e^{14}-e^{15}+e^{26}+e^{27}-e^{34}-e^{35}$\\
&$\;\;\;\;\;\;\;+2e^{48}-2e^{58}-2pe^{68}+2pe^{78}$\\
\hline
$\mL^{p,-p\neq 0}_{8,7}=\msl\ltimes A^{1,p,-p}_{5,7}$	&$\omega_{\pm}= e^{12}+e^{15}-e^{38}-e^{56}\mp e^{67}$\\
\hline
$\mL^0_{8,8}=\msl\ltimes A^{1}_{5,8}$	&$\omega_{\pm}= e^{12}+e^{15}-e^{38}-e^{56}\mp e^{67}$\\
\hline
$\mL^{0,q\neq 0}_{8,9}=\msl\ltimes A^{1,0,q}_{5,13}$	&$\omega_{\pm}= e^{12}+e^{15}-e^{38}-e^{56}\mp e^{67}$\\
\hline
$\mL_{8,16}=\msl\ltimes A^{1}_{5,15}$	&$\omega= e^{15}-e^{16}-e^{27}-e^{34}-e^{58}-e^{68}-e^{78}$\\
\hline
\multirow{2}{*}{$\mL^{p=0}_{8,17}=\msl\ltimes A^{1,0,0}_{5,7}$}	&$\omega_{\pm}= e^{12}+e^{15}-e^{38}-e^{56}\mp  e^{67}$\\
&$\omega= -e^{14}-e^{17}-e^{25}+e^{36}-e^{48}$\\
\hline
 $\mL^{p\in]-1,1]\setminus\{0\}}_{8,17}=\msl\ltimes A^{1,p,p}_{5,7}$&$\omega= -e^{14}-e^{17}-e^{25}+e^{36}-e^{48}+pe^{78}$\\ 
 \hline
  $\mL_{8,18}^{p>0}=\msl\ltimes A^{1,p,p}_{5,17}$&$\omega= pe^{15}-e^{16}-e^{27}-pe^{34}+e^{48}-p^2e^{58}-p(e^{68}+e^{78})$\\
 \hline
$\mL_{8,20}=\msl\ltimes A^{1,1,1}_{5,7}$	&$\omega= -e^{15}+\frac{1}{2} e^{17}-2e^{26}-e^{34}-\frac{1}{2} e^{36}-e^{58}-\frac{1}{6} e^{78}$\\
\hline
$\aff\oplus\R^2$	&$\omega= e^{12}+e^{15}-e^{34}-e^{56}+e^{78}$\\
\hline
$\aff\oplus\mathfrak{aff}(\R)$	&$\omega= e^{12}+e^{15}-e^{34}-e^{56}+e^{78}$\\
\hline
	\end{longtable}}
\end{theorem}

 Our main result (Theorem \ref{the1}) fall into four steps:

Let $\G$ be an eight-dimensional non-solvable Lie algebra, i.e., $\G$ admits a non-trivial Levi-Malcev decomposition $\G =N\ltimes R$, with $N$ is the Levi factor, and $R$ its radical part.
\begin{enumerate}

\item Initially, we will consider the case in which $\G$ can be decomposed. In this case, any symplectic Lie algebra of dimension less than or equal to six is either solvable or a Lie algebra of affine transformation $\aff$.  We can deduce the statement for eight-dimensional symplectic non-solvable decomposable Lie algebras. Proposition \ref{pr 0} provides the result.
\item Our study identifies the symplectic Lie algebras among the 22 non-solvable indecomposable Lie algebras listed by Turkowski \cite{T}.
The results are detailed in Lemma \ref{Lem1} and Proposition \ref{proforsym}.
\item We give a complete classification of eight-dimensional symplectic non-solvable indecomposable Lie algebras and a description of all their symplectic structures, we additionally use this
description to give some first properties. The result is given by Proposition \ref{pr5} and
their consequences.
\item As a result of the previous description, we provide a complete classification, up to symplectomorphism, of eight-dimensional  symplectic non-solvable indecomposable Lie algebras. We proceed as follows: 

In light of Proposition \ref{pr 0} and \ref{pr5}, we can distinguish two types of eight-dimensional symplectic non-solvable indecomposable Lie algebras: 
\begin{enumerate}
\item The first one admit a unique Lagrangian ideal $\J$, this type is Frobeniusian. It is
known that in the presence of an isotropic ideal $\J$, the symplectic
stucture on $\G$ induces a left-symmetric algebra (LSA) structure on
$\G/\J$. Then, we classify all LSA structures on 4-dimensional
Lie algebras of the form $\mathfrak{s} \oplus \R$, $\mathfrak{s} = \mathfrak{so}(3)$ or $\mathfrak{s} = \mathfrak{sl}_2(\R)$ (see, Theorem \ref{the 2}), which are
precisely those arising as $\G/\J$ for $\G$ an exact symplectic Lie algebra
with a Lagrangian ideal $\J$. It is known \cite{B-C} that every symplectic Lie algebra $(\G,\omega)$, which has a Lagrangian ideal $\J$ arises as a Lagrangian extension of a left-symmetric Lie algebra. Then, the first type  can be
reconstructed from the left-symmetric Lie algebra structure on $\G/\J$. The results of this part are summarized in Table \ref{tab2} and \ref{tab3}.
\item  Using the fact that the isomorphism classes of Lagrangian symplectic extensions of a left-symmetric Lie
algebra $(\h,\cdot)$ are parametrized by a suitable restricted cohomology group $H^2_{L,\rho}(\h,\h^\ast)$, with $\h\cong\G/\J$, we give  a complete classification of symplectic structures, up to symplectomorphism for the Frobeniusain Lie algebras. In Proposition \ref{pr I}, the result is described in detail.
\item As for the Lie algebras which are not exact (second type),  we show that they admit
a symplectic ideal $\mathfrak{a}$. As a consequence, such symplectic Lie algebras are semidirect products $\mathfrak{a}\rtimes\aff$, where $\aff$ is the Lie algebra
of the affine transformations of $\R^2$, which is known to admit a unique
symplectic structure. A detailed description of the result is given in Proposition \ref{pr II}.
\end{enumerate}
\end{enumerate}

This paper is organized as follows. In Section $2$, we recall some basic properties of index of Lie algebra and symplectic Lagrangian reduction. In Section $3$, we give a complete classification of eight-dimensional symplectic  non-solvable  Lie algebras and a description of all their symplectic structures, we additionally use this description to give some first properties.  In section $4$, using the results of the previous section, we give a complete classification, up to symplectomorphism, of eight-dimensional symplectic non-solvable  Lie algebras. Section $5$ is an Appendix  contains all the tables and  the details of the computations needed in the proof of some propositions.

The authors wish to thank S. El bourkadi for having read carefully the first version of
the paper and pointing out many typing and style mistakes.

\section{Preliminaries}

In this section, we recall some  notions on symplectic Lie algebras,
and Frobenius Lie algebras.

\textbf{ Index of Lie algebra}:
Let $f\in\G^*$ and $B_f$ is the associated skew-symmetric Kirillov
form defined by $B_f(x,y)=f([x,y])$. The index of Lie algebra $\G$
(see \cite{D}) is the integer invariant defined by
\begin{equation*}
\ind_\G=\min_{f\in\G^*}\dim\big({\ker B_f}\big),
\end{equation*}
with $\ker(B_f) = \{x\in\G\; :\; f\big([x,y]\big) = 0,\;\forall y\in\G\}$.

Let $\{x_1,...,x_n \}$ be a basis of $\G$. We can express the index
using the matrix $\M_\G=([x_i,x_j])_{\{1\leq i< j\leq n\}}$ as a
matrix over the ring $S(\G)$, (see \cite{D}). We have the following
characterization
\begin{equation}\label{ind}
\ind_\G=\dim\G-\rank_{R(\G)}(\M_\G),
\end{equation}
where $R(\G)$ is the quotient field of the symmetric algebra $S(\G)$.

From the above, it follows that: Let $\G$ be a Lie algebra.
Then $\ind_\G= 0$ if and only if $\G$ is Frobenius (admits an exact
symplectic form).

\textbf{Symplectic Lagrangian reduction}: The results of this
paragraph are detailed and demonstrated in \cite{B-C}. Let $(\G,\om)$ be
a symplectic Lie algebra. An ideal $\J$ of $(\G,\omega)$ is called  isotropic
  if $\J\subset\J^\perp$ with
\[\J^\perp=\{x\in\G\;:\;\om(x,y)=0,\;\forall y\in\J\}.\]
 If the orthogonal $\J^\perp$ is an ideal in $\G$ we call $\J$
a normal isotropic ideal. If $\J$ is a maximal isotropic subspace $\J$
is called a Lagrangian ideal.

Let $\J$ be a normal isotropic ideal and let $\h=\G/\J^\perp$ denote the
associated quotient Lie algebra. From $\om$ we obtain a non-degenerate
bilinear pairing  $\om_\h$ between $\h$ and $\J$, by declaring
\begin{equation*}\om_\h(\overline{x},a)=\om(x,a),\qquad\forall
\overline{x}\in\h,\;a\in\J.\end{equation*}

\begin{proposition}\label{pr 1}
The homomorphism $\om_\h\in \mathrm{Hom}(\h,\J^*)$,
$\overline{x}\longmapsto\om_\h(\overline{x},.) $, is an isomorphism and $\h$ carries a left-symmetric product
 defined by the equation
\begin{equation*}\om_\h(\overline{x}\cdot \overline{y},a)=-\om(y,[{x},a]),\qquad\forall
\overline{x},\,\overline{y}\in\h,\;a\in\J.\end{equation*}
\end{proposition}
Conversely, let $(\h,\cdot)$ be a left-symmetric algebra, since $``\cdot"$ is a left-symmetric product, it defines a Lie algebra representation. We denote by $\rho$ the dual representation. To each cocycle $\al\in Z^2_\rho(\h,\h^*)$ is associated an extension of Lie algebras
\[0\lr\h^*\lr\G_{\rho,\al}\lr\h\lr0,\]
with $\G_{\rho,\al}=\h\oplus \h^*$ and the non-zero Lie brackets are defined by
	\begin{align}\label{bra}
		\begin{split}
	[x,y]_{\G_{\rho,\al}}& =[x,y]_\h+\al(x,y),\qquad x, y\in\h, \\
	{[x,\xi]_{\G_{\rho,\al}}} &=\rho(x)\xi, \qquad\qquad x\in\h\;\;and\,\; \xi\in\h^*.
	\end{split}
\end{align}
We let $\om$ be the non-degenerate alternating two-form on $\G_{\rho,\al}$, which is defined by the dual pairing of $\h$ and $\h^*$ (i.e., $\om(x_1,x_2)=\om(\xi_1,\xi_2)=0$ and $\om(x_1,\xi_1)=\xi_1(x_1)$, $\forall x_1,x_2\in\h$, $\forall \xi_1,\xi_2\in\h^*$).
\begin{proposition}\label{pr2}
The form $\om$ is symplectic for the Lie-algebra $\G_{\al,\rho}$ if and only if
\begin{eqnarray}\label{sym}
\al(x,y)(z)+\al(y,z)(x)+\al(z,x)(y)=0,\qquad\forall x,y,z\in\h.
\end{eqnarray}
\end{proposition}
The symplectic Lie algebras $(\G,\om)$ which admit a Lagrangian ideal $\J$ can be constructed from the left symmetric algebra $(\h=\G/\J^\perp,\cdot)$.

We now construct for any left-symmetric Lie algebra $(\h, \cdot)$ a cohomology group, which describes all Lagrangian extensions of $\h$ with associated left-symmetric Lie algebra $(\h, \cdot)$.

First, we define Lagrangian one- and two-cochains on $\h$ as 
\begin{eqnarray*}
	\mathcal{C}^1_L(\h,\h^\ast)&=&\big\{\varphi\in\mathcal{C}^1(\h,\h^\ast) : \varphi(x)(y)-\varphi(y)(x)=0,~\text{for all}~ x,y\in\h\big\},\\
	\mathcal{C}^2_L(\h,\h^\ast)&=&\big\{\alpha\in\mathcal{C}^2(\h,\h^\ast) : \alpha~\text{satisfies}~(\ref{sym})\big\}.
\end{eqnarray*}
 Denote by $\partial_\rho = \partial_\rho^i$ the corresponding coboundary operators for cohomology with $\rho$-coefficients.

Let $Z^2_{L,\rho}(\h,\h^\ast)=\mathcal{C}^2_L(\h,\h^\ast)\cap Z^2_\rho(\h,\h^\ast)$  denote the space of Lagrangian
cocycles. We now define the Lagrangian extension cohomology group for the
left-symmetric Lie algebra $(\h,\cdot)$ as
\begin{center}
	$H^2_{L,\rho}(\h,\h^\ast)=\dfrac{Z^2_{L,\rho}(\h,\h^\ast)}{\partial_\rho\mathcal{C}^1_L(\h,\h^\ast)}$.
\end{center}
\begin{remark}\label{kernel}
	By construction there is a natural map from $H^2_{L,\rho}(\h,\h^\ast)$ to the ordinary Lie algebra cohomology group $H^2_\rho(\h, \h)$. Note that this map need not be injective, in general. The kernel $\kappa_L$ of the natural map
	\begin{equation}\label{map}
	H^2_{L,\rho}(\h,\h^\ast)\longrightarrow H^2_\rho(\h, \h^\ast)
	\end{equation}
	 is given by
	\begin{center}
		$\kappa_L=\dfrac{B^2_\rho(\h,\h^\ast)\cap Z^2_{L,\rho}(\h,\h^\ast)}{B^2_{L,\rho}(\h, \h^\ast) }$,
	\end{center}
	where $B^2_\rho(\h, \h^\ast) = \{\partial_\rho \alpha ~|~ \alpha \in\mathrm{Hom}(\h, \h^\ast)\}$ is the set of ordinary two-coboundaries with $\rho$-coefficients and $B^2_{L,\rho}(\h, \h^\ast) = \{\partial_\rho \alpha ~|~\alpha \in\mathcal{C}^1_L(\h, \h^\ast)\}$ is the set of two-coboundaries for Lagrangian extension cohomology.
\end{remark}

The following proposition will subsequently play an essential role in the classification of symplectic structures on eight-dimensional Frobeniusian  non-solvable  Lie algebras.
\begin{proposition}\textsc{\cite{B-C}}\label{pr 3}
There is a one-to-one correspondence between the classes of isomorphisms of symplectic Lie algebras which admit a Lagrangian ideal and the triples $(\h,\cdot, [\al])$, with  $\al\in H^2_{L,\rho}(\h,\h^\ast)$.
\end{proposition}

\textbf{Non-solvable six-dimensional case}: The non-solvable six-dimensional Lie algebras are classified by Turkowski \cite{T}, up to isomorphism, into four Lie algebras.
\[\mL_{6,3}=\msl\ltimes A_{3,3},\quad\mL_{6,1}=\mso\ltimes3\mL_1,\quad
	\mL_{6,2}=\msl\ltimes A_{3,1},\quad	
	\esp \mL_{6,4}=\msl\ltimes 3\mL_1 
	.\]
The last three Lie algebras do not support symplectic structures (e.g. they satisfy $[\G,\G]=\G$).
The first Lie algebra $\mL_{6,3}=\msl\ltimes A_{3,3}$ is the Lie algebra of affine transformation  $\aff$ which admits a unique symplectic structure (see \cite{A}). In the sequel, the symplectic Lie algebra $(\aff,\om)$ is represented by $\aff=\langle e_1,\ldots,e_6\rangle$ with the brackets
\[[e_1,e_2]=2e_2,\, [e_1,e_3]=-2e_3,\, [e_1,e_4]=e_4,\, [e_1, e_5]=-e_5, \,[e_2,e_3]=e_1\]\[ [e_2,e_5] = e_4,\, [e_3,e_4] = e_5,\, [e_4,e_6] = e_4,\,[e_5,e_6] = e_5,\]
the unspecified brackets are either zero or given by antisymmetry.
The symplectic form is given by
\[\om=e^{12}+e^{15}-e^{34}-e^{56}.\]
For more details on properties symplectic structure on Lie algebra of affine transformation $\aff$, see \cite{B-M}.

Any symplectic Lie algebra of dimension less than or equal to six is either solvable or Lie algebra of affine transformation $\aff$. We can now deduce the following statement for eight-dimensional symplectic non-solvable decomposable Lie algebras. 

\begin{proposition}\label{pr 0}
	Let $\G$ be an eight-dimensional symplectic non-solvable decomposable Lie algebra. Then $\G$ is symplectomorphically equivalent to one of the following Lie algebras equipped with a symplectic form as follows$:$
		\begin{align*}
\aff\oplus\mathfrak{aff}(\R),&\qquad\om=e^{12}+e^{15}-e^{34}-e^{56}+e^{78},\\	
\aff\oplus\R^2,&\qquad\om=e^{12}+e^{15}-e^{34}-e^{56}+e^{78}.
	\end{align*}
\end{proposition}
In what follows, the Lie algebras are indecomposable.

\section{ Eight dimensional symplectic non-solvable Lie algebras}
In this section, we detect among the $22$ indecomposable non-solvable  Lie algebras listed by Turkowski \cite{T} those which are symplectic (Frobeniusian or not Frobeniusian) and those which do not support a symplectic structure.
\begin{lemma}\label{Lem1}
                The following Lie algebras
        \[\mL_{8,1}^p,\;\mL_{8,2},\;\mL_{8,4}^0,\;\mL_{8,5},\;\mL_{8,6},\;\mL^p_{8,8},\;\mL^{p,q\neq0}_{8,9},\;\mL_{8,10}^p,\;\mL_{8,11},\;\mL_{8,12}^{p},\]
        \[\mL_{8,13}^{\epsilon=\pm1},\;\mL_{8,14},\;\mL_{8,15},\;\mL^{-1}_{8,17},\;\mL^0_{8,18},\;\mL_{8,19},\;\mL_{8,21},\;\mL_{8,22},\]
do not support a symplectic structure $($with $p\not=0)$.
        \end{lemma}

\begin{proof}
        We distinguish between two types of Lie algebras:
\begin{enumerate}
        \item The first type:
        \[\mL_{8,2},\;\mL_{8,4}^0,\;\mL_{8,5},\;\mL_{8,6},\;\mL^{\epsilon=\pm1}_{8,13},\;\mL_{8,14},\;\mL_{8,15},\;\mL^{-1}_{8,17},\;\mL_{8,18}^0,\;\mL_{8,19},\;\mL_{8,21},\;\mL_{8,22},\]
        which are unimodular ($\tr(\ad_x)=0$) and non-solvable therefore
cannot be symplectic.
        \item The second type:
        \[\mL^p_{8,1},\;\mL_{8,8}^p,\;\mL^{p,q\neq0}_{8,9},\;\mL^p_{8,10},\;\mL_{8,11},\;\mL^p_{8,12},\]
        which have  trivial second  cohomology group (i.e., $H^2(\G)=0$),
if this kind of Lie algebras supports symplectic structures they will
be exact. The problem of the non-existence of the symplectic structure
therefore boils down to the calculation of the index.   Indeed, a direct
calculation using the formula (\ref{ind}) gives us that the index of
the second type Lie algebras is two.\end{enumerate}
\end{proof}

\begin{proposition}\label{proforsym}
	
	Eight-dimensional  non-solvable Lie algebras that support a symplectic structure are isomorphic to one of the following Lie algebras
       \[\mL_{8,3},\;\mL^{p\not=0}_{8,4},\;\mL^{p,-p\neq0}_{8,7},\;\mL^0_{8,8},\;\mL^{0,q\neq0}_{8,9},\;\mL_{8,16},\;\mL^{p\in]-1,1]}_{8,17},\;\mL_{8,18}^{p>0},\;\mL_{8,20}.\]
In addition, 
 \begin{itemize} 	
 \item The Lie algebras $\mL_{8,3},\;\mL^{p\not=0}_{8,4},\;\mL_{8,16},\;\mL^{p\in]-1,1]}_{8,17},\;\mL_{8,18}^{p>0},\;\mL_{8,20}$ are Frobeniusian.   Further admitting  only exact symplectic structures $($except for $\mL^{0}_{8,17})$.
 \item The Lie algebras      $\mL^{p,-p\neq0}_{8,7},\;\mL^0_{8,8},\;\mL^{0,q\neq0}_{8,9},$
are not Frobeniusian.
\end{itemize}           
\end{proposition}

\begin{proof}
For Lie algebras
\[\mL_{8,3},\;\mL^{p\not=0}_{8,4},\;\mL^{p,-p\neq0}_{8,7},\;\mL_{8,16},\;\mL^{p\in]-1,1]\setminus\{0\}}_{8,17},\;\mL_{8,18}^{p>0},\;\mL_{8,20}.\]
On the one hand, their second cohomology group is zero, so any symplectic
structure is exact, and on the other hand their index is zero.

For the following Lie algebras
\[\mL^{p,-p\neq0}_{8,7},\;\mL^0_{8,8},\;\mL^{0,q\neq0}_{8,9},\;\mL^0_{8,17},\]
by studying case by case (see Proposition \ref{pr5} ) and finding their
symplectic structures. In addition, note that the index of Lie algebras 
\[\mL^{p,-p\neq0}_{8,7},\;\mL^0_{8,8},\;\mL^{0,q\neq0}_{8,9}\]
 is not zero, these symplectic structures are therefore not Frobeniusian.
\end{proof}
\begin{remark}
After a short calculation we obtain that, the Lie algebra $\mL_{8,7}^{p,q\neq0}$ has a symplectic structure if and only if $q=-p$. In the same way, $p=0$ for Lie algebra $\mL_{8,9}^{p,q\neq0}$.
\end{remark}
\begin{remark}
The Lie algebra $\mL^0_{8,17}$ is the only Frobeniusian Lie algebra with non-trivial second cohomology group.	
\end{remark}	

\begin{proposition}\label{pr5} Let $(\G,\om)$ be an eight-dimensional   symplectic indecomposable non-solvable
 real Lie algebra. Then
        $(\G,\om)$ is isomorphic to one of the following symplectic Lie algebras:
 \begin{enumerate}     
        \item[$\bullet$] $\mL_{8,3}=\mso\ltimes A_{5,7}^{1,1,1}$.\quad For~
 $a^2_{26} + a^2_{34} + a^2_{36} + a^2_{37}\not=0$,
                \begin{align*}
        \om&=a_{12}e^{12}+a_{13}e^{13}+a_{23}e^{23}
        +a_{26} (e^{14} + e^{26}-e^{35} + 2e^{78} )
        +a_{34} (e^{15}-e^{27} + e^{34} + 2e^{68} )\\
        &+a_{36} (e^{17} + e^{25} + e^{36} - 2e^{48})
        +a_{37} (-e^{16} + e^{24} + e^{37} + 2e^{58}).
        \end{align*}
 \item[$\bullet$] $\mL^{p\not=0}_{8,4}=\mso\ltimes
A_{5,17}^{1,p,p}$.\quad For~ $a_{34}^2 + a_{35}^2 + a_{36}^2 +
a_{37}^2\neq 0$,
 \begin{align*}
 \om&=a_{12}e^{12}+a_{13}e^{13}+a_{23}e^{23}+a_{34}(e^{15}-e^{27}+e^{34}-2e^{48}+2pe^{68})\\&+a_{35}(-e^{14}-e^{26}+e^{35}+2e^{58}-2pe^{78})+a_{36}(e^{17}+e^{25}+e^{36}-2pe^{48}-2e^{68})\\
 &+a_{37}(-e^{16}+e^{24}+e^{37}+2pe^{58}+2e^{78}).
 \end{align*}
 \item[$\bullet$] $\mL^{p,-p\neq 0}_{8,7}=\msl\ltimes A^{1,p,-p}_{5,7}.$\quad For~ $a_{67}(a_{12}a_{58}^2 + a_{13}a_{48}^2 + 2\,a_{23}a_{48}a_{58})^2\neq 0$,
   \begin{align*}
 \om &=a_{12}e^{12}+a_{13}e^{13}+a_{23}e^{23}+a_{67}e^{67}+a_{68}e^{68}+a_{78}e^{78}+a_{48}(e^{14}+e^{25}+e^{48})\\&+a_{58}(-e^{15}+e^{34}+e^{58}).
\end{align*}
 \item[$\bullet$] $\mL^0_{8,8}=\msl\ltimes A^{1}_{5,8}.$ \quad For~ $a_{6 7}(a_{1 2}a^2_{3 4} + a_{1 3}a^2_{2 5} + 2a_{2 3}a_{2 5}a_{3 4})\neq 
0$,
 \begin{align*}
  \om &=a_{12}e^{12}+a_{13}e^{13}+a_{23}e^{23}+a_{67}e^{67}+a_{68}e^{68}+a_{78}e^{78}+a_{25}(e^{14}+e^{25}+e^{48})\\&+a_{34}(-e^{15}+e^{34}+e^{58}).
\end{align*}
 \item[$\bullet$] $\mL^{0,q\neq 0}_{8,9}=\msl\ltimes A^{1,0,q}_{5,13}.$\quad For~ $a_{6 7}(a_{1 2}a^2_{5 8} + a_{1 3}a^2_{4 8} + 2a_{2 3}a_{4 8}a_{5 8})\neq 0$,
 \begin{align*}
  \om&=a_{12}e^{12}+a_{13}e^{13}+a_{23}e^{23}+a_{67}e^{67}+a_{68}e^{68}+a_{78}e^{78}+a_{48}(e^{24}+e^{25}+e^{48})\\&
+a_{58}(-e^{15}+e^{58}+e^{34}). 
 \end{align*} 
 \item[$\bullet$] $\mL_{8,16}=\msl\ltimes A^{1}_{5,15}.$\quad For~ $a_{25}a_{36}-a_{27}a_{34}\neq 0$,
 \begin{align*}
 \om&=a_{12}e^{12}+a_{13}e^{13}+a_{23}e^{23}+a_{25}(e^{14}+e^{25}+e^{48}+e^{68})+a_{27}(e^{16}+e^{27}+e^{68})&\\&+a_{34}(-e^{15}+e^{34}+e^{58}+e^{78})+a_{36}(-e^{17}+e^{36}+e^{78}).
 \end{align*} 
 \item[$\bullet$] $\mL^{p=0}_{8,17}=\msl\ltimes A^{1,0,0}_{5,7}.$~ For~  $(a_{13}a_{67} - a^2_{36})a^2_{25}+ a^2_{3 4}(a_{1 2}a_{6 7} - a^2_{2 7}) + 2a_{3 4}(a_{23}a_{6 7} + a_{27}a_{3 6})a_{2 5}\neq 0$,
 \begin{align*}
 \om&= a_{12}e^{12}+ a_{13}e^{13}+a_{23}e^{23}+a_{27}e^{27}
+a_{36}e^{36}+a_{67}e^{67}+a_{27}(e^{16}+e^{27})+a_{25}(e^{14}+e^{25}+e^{48})\\&
+a_{34}(-e^{15}+e^{34}+e^{58})+a_{36}(-e^{17}+e^{36}).
 \end{align*} 
 \item[$\bullet$] $\mL^{p\in]-1,1]\setminus\{0\}}_{8,17}=\msl\ltimes A^{1,p,p}_{5,7}.$\quad For~ $a_{14}a_{17} - a_{15}a_{16}\neq 0,$
 \begin{align*}
 \om&=a_{12}e^{12}+a_{13}e^{13}+a_{14}(e^{14}+e^{25}+e^{48})+a_{15}(e^{15}-e^{34}-e^{58})+a_{16}(e^{16}+e^{27}+pe^{68})\\&
+a_{17}(e^{17}-e^{36}-pe^{78}).
\end{align*}
 \item[$\bullet$] $\mL_{8,18}^{p>0}=\msl\ltimes A^{1,p,p}_{5,17}.$ \quad For ~$a_{25}a_{36}- a_{27}a_{34}\neq 0$,
 \begin{align*}
\om&=a_{12}e^{12}+a_{13}e^{13}+a_{23}e^{23}+a_{25}(e^{14}+e^{25}+pe^{48}+e^{68})+a_{34}(-e^{15}+e^{34}+pe^{58}+e^{78})\\&
+a_{27}(e^{16}+e^{27}-e^{48}+pe^{68})+a_{36}(-e^{17}+e^{36}-e^{58}+pe^{78}). 
 \end{align*} 
 \item[$\bullet$] $\mL_{8,20}=\msl\ltimes A^{1,1,1}_{5,7}.$\quad For~$a^2_{15}(4a_{15}a_{36} - 3a^2_{16})+ a^2_{25}a^2_{36}+ a_{16}a_{25}(6a_{15}a_{36} - 4a^2_{16}) \neq 0.$
\begin{align*}
\om&= a_{12}e^{12}+a_{13}e^{13}+a_{23}e^{23}+a_{15}(e^{15}+2e^{26}+e^{34}+e^{58})+a_{16}(e^{16}-e^{27}-2e^{35}-e^{68})\\&
+a_{25}(e^{14}+e^{25}+\frac{1}{3} e^{48})+a_{36}(-e^{17}+e^{36}+\frac{1}{3} e^{78}).
\end{align*}
 \end{enumerate}       
\end{proposition}
\begin{proof}
	The proof follows by working on each Lie algebra. We first compute the 2-cocycles, (i.e., the 2-forms which verifies (\ref{cocy})) the next step is to compute the rank of $\om$. If $\om$ has maximal rank, that is, $\wedge^4\omega\neq 0$ then $\G$ will be endowed with a symplectic structure.
\end{proof}

\begin{remark}
	\begin{enumerate}
		\item Every eight-dimensional non-solvable symplectic Lie algebra has a trivial center, which is not true in the general case. For example, the center of the Lie algebra $\mL_{8,2}$ is not trivial.
		\item  Every eight-dimensional symplectic non-solvable   Lie algebra has the subalgebra $\J=\langle e_4,e_5,e_6,e_7\rangle$ as an  abelian ideal,  which is not true in the general case, see for example $\mL_{8,2}$. 
		\item For the Lie algebras $\mL^{p,-p\neq0}_{8,7},\;\mL^0_{8,8},\;\mL^{0,q\neq0}_{8,9}$ and $\mL^0_{8,17}$ the subalgebra $\mathfrak{a}=\langle e_6,e_7\rangle$ is an abelian ideal.
	\end{enumerate}	
\end{remark}
\begin{corollary}\label{co1}
 For all  eight-dimensional   non-solvable Frobeniusian Lie algebras, $\J=\langle e_4,e_5,e_6,e_7\rangle$  is a unique Lagrangian ideal.
\end{corollary}

\begin{proof}
        We will give the proof for the symplectic Lie algebra $\mL_{8,3}$, since
all cases must be
        treated in the same way.
        Let $x=x_1e_1+\cdots+x_8e_8\in\mL_{8,3}$, we have $\om(x,e_i)=0$ for
$i\in\{4,5,6,7\}$ is equivalent to
        \[(S)\qquad\left\{
        \begin{array}{l}
        x_1a_{26} + x_2a_{37} + x_3a_{34} + 2x_8a_{36}=0\\
        x_1a_{34}+ x_2a_{36} - x_3a_{26} - 2x_8a_{37}=0\\
        x_1a_{37}- x_2a_{26} - x_3a_{36} + 2x_8a_{34}=0\\
        x_1a_{36}- x_2a_{34} + x_3a_{37} - 2x_8a_{26}=0.
        \end{array}
        \right.\]
        The determinant of the system $(S)$ is $-2(a^2_{26} + a^2_{34} +
a^2_{36} + a^2_{37})^2$ which is nonzero ($\om$ is not degenerate), so
$x=x_4e_4+x_5e_5+x_6e_6+x_7e_7\in\J$. For the uniqueness of $\J$, it suffices to notice that for dimension reasons, $\J$ coincides with the nilradical of $\mL_{8,3}$.
\end{proof}
From the general form of the symplectic forms listed in Proposition \ref{pr5} it is easy to see the following corollary.
\begin{corollary}\label{co2}
	The symplectic Lie algebras $\mL^{p,-p\neq0}_{8,7},\;\mL^0_{8,8},\;\mL^{0,q\neq0}_{8,9}$ and $\mL^0_{8,17}$ for $a_{67}\not=0$ admit $\mathfrak{a}=\langle e_6,e_7\rangle$ as a symplectic ideal.
\end{corollary}

\section{Classification of symplectic structures on eight-dimensional non-solvable  Lie algebras}
Let $\G=\langle e_1,\ldots,e_8 \rangle$, be an eight-dimensional symplectic non-solvable   Lie algebra. $\G$ admits a non-trivial Levi-Malcev decomposition $\G=\mathfrak{s}\ltimes \mathfrak{r}$, where $\mathfrak{s}$ its  semisimple part ($\mso$ or $\msl$), and $\mathfrak{r}$ its  radical part. The Lie algebra $\mso$ (resp. $\msl$)  is defined by the non-vanishing structure constants	
\[[e_1,e_2]=e_3,~[e_1,e_3]=-e_2,~[e_2,e_3]=e_1.\]
(resp.
\[[e_1,e_2]=e_3,~[e_1,e_3]=-2e_1,~[e_2,e_3]=2e_2.)\]	
Put  $\J=\langle e_4,e_5,e_6,e_7\rangle$ and $\mathfrak{a}=\langle e_6,e_7\rangle$. Based on Corollaries \ref{co1} and \ref{co2}, we can distinguish two types of  eight-dimensional symplectic non-solvable  Lie algebras.

\begin{itemize} 	
	\item[Type I.] The one with a Lagrangian ideal $\J$, this type is Frobeniusian and can be reconstructed from the left symmetric Lie algebra $(\h=\G/\J,\cdot)$.
	\item[Type II.]  The one with a symplectic ideal $\mathfrak{a}$, this type admits decomposition $\aff \ltimes \mathfrak{a}$.
\end{itemize} 	
\subsection{Type I}

Let $(\G,\omega)$ be an  eight-dimensional Frobeniusian non-solvable Lie algebra. From Corollary $\ref{co1}$, $(\G,\omega)$ has a Lagrangian ideal $\J=\langle e_4,e_5,e_6,e_7\rangle$. In this case, 
	$\h=\G/\J$ is isomorphic to  $\mathfrak{s}\oplus\R e_4^\prime$, where $\mathfrak{s}$ is a simple Lie algebra ($\mso$ or $\msl$), and $e_4^\prime$ is a central element in $\h$. Using Proposition \ref{pr 3} the problem of classification of Frobeniusian structures on eight-dimensional non-solvable Lie algebra is reduced to the classification of left-symmetric structures on $\h$. Recall that, two LSA $(\h,\mL)$ and $(\h, \tilde{\mL})$ are isomorphic if and only if there exists  $\psi \in\mathrm{Aut}(\h)$ such that, $\tilde{\mathrm{L}}(x)=\psi\circ\mathrm{L}(\psi^{-1}(x))\circ\psi^{-1}$ for all $x\in\h$.
	
To avoid any confusion, let us take $e_4^\prime=e_4$  in the next calculation.
\begin{proposition} The left-symmetric structures on $\h$ has a right-identity $($i.e., $\mathrm{R}(e) = \mathrm{Id}_\h)$ noted $e=\lambda_1 e_1+\lambda_2 e_2+\lambda_3 e_3+\lambda_4 e_4$ with $\lambda_4\neq 0$. Moreover, we have
\begin{enumerate}
	\item For $\h=\mso\oplus\mathbb{R} e_4$.
	We may assume that $e = \mu e_1+e_4$ or $e = e_4$.
	\item For $\h=\msl\oplus\mathbb{R}e_4$.
	We may assume that   $e =  e_1+e_4$, $e=\lambda e_3+e_4$, $e=e_1+\nu e_2+e_4$ or $e=e_4$.
\end{enumerate}	
\end{proposition}
\begin{proof} 
The	left-symmetric product in $\h$ is
	defined by the equation
	\begin{equation*}\om_\h(\overline{x}\cdot \overline{y},a)=-\om(y,[x,a]),\qquad\forall
	\overline{x},\,\overline{y}\in\h,\;a\in\J.\end{equation*}
If  $\G$ is a Frobenius Lie algebra. Then $\om(x,y)=-\beta([x,y]),\forall x,y\in\G$ for some $\beta\in\G^*$, which implies that there
exists a unique element $e\in\G$ such that $\om(e,z)=\beta(z)$, for any $z\in\G$. Hence we have
\[\om_\h(\overline{x}\cdot \overline{e},a)=-\om(e,[x,a])=-\beta([x,a])=\om_h(\overline{x},a).\]
$\forall a\in\J$, form Proposition \ref{pr 1}; therefore $\overline{x}\cdot\overline{e}=\overline{x}$ and $\overline{e}$ is a right identity. In what follows, we use the identification $\h\simeq\mathfrak{s}\oplus\R e_4$ and set $e=\lambda_1 e_1+\lambda_2 e_2+\lambda_3 e_3+\lambda_4 e_4$ or to simplify $e=(\lambda_1,\lambda_2,\lambda_3,\lambda_4)$. We have $\lambda_4\not=0$, otherwise
 \[0=\mathrm{tr}\big(\mathrm{R}(\lambda_1 e_1)\big)+\mathrm{tr}\big(\mathrm{R}(\lambda_2 e_2)\big)+\mathrm{tr}\big(\mathrm{R}(\lambda_3 e_3)\big)=\mathrm{tr}(\mathrm{R}(e))=4.\]
Indeed, $\mathrm{tr}\big(\mathrm{L}([x, y])\big) = \mathrm{tr}([\mathrm{L}(x), \mathrm{L}(y)])$, $[\mathfrak{s}, \mathfrak{s}] = \mathfrak{s}$ and 
	\begin{equation}
	\mathrm{tr}(\mathrm{R}(x)) = \mathrm{tr}(\mathrm{L}(x))-\mathrm{tr}(\mathrm{ad}(x)) = 0. 
	\end{equation}

	Note that, if $(\h, \mathrm{L})$ and $(\h,\tilde{\mathrm{L}} )$ are isomorphic LSA's, then $\mathrm{R}(e)=\mathrm{Id}_\h$ implies $\tilde{\mathrm{R}}(\psi(e))=\psi\circ\mathrm{R}(e)\circ\psi^{-1}=\mathrm{Id}_\h$, i.e., the LSA $(\h,\tilde{\mathrm{L}})$ has right identity $\psi(e)$.
	
{\textbf{For $\h=\mso\oplus\mathbb{R} e_4:$}}
	
	 Let $e=(\lambda_1,\lambda_2,\lambda_3,\lambda_4)$. First we may assume $\lambda_2=0$. Otherwise,  let $\theta\in\mathbb{R}$ be a root of  $\pm\sqrt{-\theta^{2}+1}\lambda_{1}+\theta\lambda_{2}=0$, if $\lambda_2>0$ (resp. $\lambda_2<0$) we choose $\psi^-_1$ (resp. $\psi^+_1$) and we have
	\begin{center}
		$\psi^{\pm}_1(e)=\left( \sqrt{\lambda_1^2+\lambda_{2}^2},0,\lambda_{3},\eta\lambda_{4}
		\right) 
		$,
	\end{center}
	where,
	\begin{center}
		$\psi^\pm_1=\left( \begin {array}{cccc} \theta&\mp\sqrt {-\theta^{2}+1}&0&0
		\\ \noalign{\medskip}\pm\sqrt{-\theta^{2}+1}&\theta&0&0
		\\ \noalign{\medskip}0&0&1&0\\ \noalign{\medskip}0&0&0&\eta
		\end {array} \right)\in\mathrm{Aut}(\mso\oplus\mathbb{R} e_4).
		$
	\end{center}
	\textbf{Case $1:$} $\lambda_1=0$. If $\lambda_3=0$ then $\psi^{\pm}_1(e)=e_4$ with $\eta=\frac{1}{\lambda_{4}}$ (note that $e\neq 0$). If
	$\lambda_3 \neq 0$ then it follows $\lambda_4 \neq 0$, otherwise $0 =\mathrm{tr}(\mathrm{R}(\lambda e_3)) =\mathrm{tr}( \mathrm{R}(e)) = 4$, contradiction. Then $\psi^\pm_1(e) =  \lambda_3 e_3 + e_4$ with  $\eta=\frac{1}{\lambda_{4}}$.\\
	\textbf{Case $2:$} $\lambda_1\neq 0$.  Let  $\theta\in\mathbb{R}$ be a root of $\theta\lambda_1+\sqrt{-\theta^2+1}\lambda_3=0$, if $\lambda_3>0$ (resp. $\lambda_3<0$) we choose $\psi^-_2$ (resp. $\psi^+_2$), and we have 
	$\psi^\pm_2(e)=\left( 0,0,\sqrt{\lambda_1^2+\lambda_3^2},\eta\lambda_4\right)$
	and we are back to case $1$, where,
	\begin{center}
		$\psi^\pm_2=\left( \begin {array}{cccc} \theta&0&\mp\sqrt{-\theta^{2}+1}&0
		\\ \noalign{\medskip}0&1&0&0\\ \noalign{\medskip}\pm\sqrt {-{\theta}^
			{2}+1}&0&\theta&0\\ \noalign{\medskip}0&0&0&\eta\end {array}
		\right)\in\mathrm{Aut}(\mso\oplus\mathbb{R} e_4).
		$
	\end{center}
	Note that we can take $\lambda_3>0$, indeed $\psi^{+}_2(0,0,\lambda_3,1)=(0,0,-\lambda_3,1)$, with $\theta=-1$ and $\eta=1$. We have also, $\psi_2^-(0,0,\lambda_3,1)=(\lambda_3,0,0,1)$, with $\theta=0$, and $\eta=1$.

\textbf{For $\h=\mathfrak{sl}_2(\mathbb{R})\oplus\mathbb{R}e_4:$}

 Let $\psi_1, \psi_2\in\mathrm{Aut}(\mathfrak{sl}_2(\mathbb{R})\oplus\mathbb{R}e_4)$ be two automorphisms respectively defined by
	
	\begin{center}
		$\psi_1:=\left( \begin {array}{cccc} \theta_1&\theta_2&2\,\sqrt {-\theta_1\,
			\theta_2}&0\\ \noalign{\medskip}{\frac {1}{{\theta_2}^{2}} \left( -2\,
			\sqrt {-\theta_1\,\theta_2}\theta_3-\theta_1\,\theta^{2}_3+\theta_2
			\right) }&-{\frac {\theta^{2}_3}{\theta_2}}&2\,{\frac {\theta_3\,
				\left( -\sqrt {-\theta_1\,\theta_2}\theta_3+\theta_2 \right) }{{\theta_2}^
				{2}}}&0\\ \noalign{\medskip}{\frac {1}{\theta_2} \left( \theta_1\,\theta_3+\sqrt {-\theta_1\,\theta_2} \right) }&\theta_3&{\frac {1}{\theta_2}
			\left( 2\,\sqrt {-\theta_1\,\theta_2}\theta_3-\theta_2 \right) }&0
		\\ \noalign{\medskip}0&0&0&\theta_4\end {array} \right) 
		$,
	\end{center}
	\begin{center}
		$\psi_2:=\left( \begin {array}{cccc} \theta_1&0&0&0\\ \noalign{\medskip}-{
			\frac {\theta^{2}_2}{\theta_1}}&\frac{1}{\theta_1}&-2\,{\frac {\theta_2}{
				\theta_1}}&0\\ \noalign{\medskip}\theta_2&0&1&0\\ \noalign{\medskip}0&0&0
		&\theta_3\end {array} \right).
		$
	\end{center}
	\textbf{Case $1:$} $\lambda_1=0$. If $\lambda_3\neq 0$, then $\psi_2(e)=\lambda_3e_3+e_4$, with  $\theta_2=\frac{\lambda_2}{2\lambda_3}$ and $\theta_3=\frac{1}{\lambda_4}$. If $\lambda_3=0$,  we have $\psi_2(e)=e_2+e_4$, with $\theta_1=\lambda_2$ and $\theta_3=\frac{1}{\lambda_4}$ (if $\lambda_2\neq 0$) and $\psi_2(e)=e_4$ with $\theta_3=\frac{1}{\lambda_4}$ (if $\lambda_2=0$).\\\\
	\textbf{Case $2:$} $\lambda_1\neq 0$. Firstly, we have $\psi_2(e)=(1,\lambda_1\lambda_2+\lambda^2_3,0,1)$, with $\theta_1=\frac{1}{\lambda_1}$, $\theta_2=-\frac{\lambda_3}{\lambda_1}$ and $\theta_3=\frac{1}{\lambda_4}$.
	On the other hand,  let's now set $e=e_1+\nu e_2+e_4$. Firstly, we have $\psi_1(0,0,\lambda_3,1)=(0,0,-\lambda_3,1)$ with $\theta_1=0, \theta_3=0$ and $\theta_4=1$. If $\nu>0$, we have $\psi_1(e)=(0,0,\sqrt{\nu},1)$, with $\theta_1=-\nu,\theta_2=1, \theta_3=\frac{1}{2\sqrt{\nu}}$, and we are back to case 1.   Note that,   $\psi_1(0,1,0,1)=(1,0,0,1)$ with $\theta_3 = 0, \theta_2 = 1, \theta_4 = 1$.	Which completes the proof.
\end{proof}

The following theorem classifies real left-symmetric products over $\h=\mathfrak{s}\oplus\R e_4$, where $\mathfrak{s}$ is a simple Lie algebra ($\mso$ or $\msl$). In \cite{Bu} Burde classified the complex left-symmetric products for $\mathfrak{sl}_2(\mathbb{C})\oplus\mathbb{C}$.

\begin{theorem}\label{the 2}
	Let  $\h=\mathfrak{s}\oplus\R e_4$, where $\mathfrak{s}$ is a simple Lie algebra $(\mso$ or $\msl)$. Then the left symmetric product in $\h$ is listed in Table $\ref{tab2}$.	

\end{theorem}
\begin{proof}
	The left symmetry product is given by $64$ structure constants via $\mathrm{L}(e_1)$, $\mathrm{L}(e_2)$, $\mathrm{L}(e_3)$, $\mathrm{L}(e_4)$.  The
	condition $[x, y] = x\cdot y-y\cdot x$ determines $24$ structure constants by linear equations. The identity $(\ref{ass})$ gives us quadratic equations  in the structure constants, which are generally difficult to solve. The existence of a non-central right-identity,  will simplify the matter considerably. We have
	\begin{equation}\label{simplifyequa}
	[\mathrm{L}(e_4),\mathrm{ad}(e)]=[\mathrm{L}(e_4),\mathrm{L}(e)]=\mathrm{L}[e_4,\,e]=0.
	\end{equation}
	Note that $\mathrm{tr}(\mathrm{R}(s))=0$ for all $s\in\mathfrak{s}$. Indeed, $\mathrm{tr}\big(\mathrm{L}([x, y])\big) = \mathrm{tr}([\mathrm{L}(x), \mathrm{L}(y)])$, $[\mathfrak{s}, \mathfrak{s}] = \mathfrak{s}$ and 
	\begin{equation}
	\mathrm{tr}(\mathrm{R}(x)) = \mathrm{tr}(\mathrm{L}(x))-\mathrm{tr}(\mathrm{ad}(x)) = 0. 
	\end{equation}
\textbf{	I. Algebras with $e=e_1+e_4$ or $e=\alpha e_3+e_4$}.

	\textbf{Case 1 :  $\mathfrak{s}=\mathfrak{sl}_2(\mathbb{R})$.} Algebras with $e=e_1+e_4$. Using (\ref{simplifyequa}), we get $[\mathrm{L}(e_4),\mathrm{ad}(e_1)]=0$ and $\mathrm{R}(e_1)+\mathrm{R}(e_4)=\mathrm{Id}_\h$ and the fact that $\mathrm{tr}(\mathrm{R}(s))=0$ for all $s\in\mathfrak{sl}_2(\mathbb{R})$. The
	remaining LSA-structure equations then are almost trivial. It is easy to see that they have
	a unique solution, which is given by the algebra $\h_1$.

	Algebras with $e=\lambda e_3+e_4$. Assume first that $\lambda> 0$. Then $\lambda\mathrm{R}(e_3)+\mathrm{R}(e_4)=\mathrm{Id}_\h$ and $[\mathrm{L}(e_4),\mathrm{ad}(e_3)]=0$. 
	This determines further  structure constants by linear equations. It is easy to solve the remaining equations and to obtain the algebra $\h_2^{\lambda>0}$.

	\textbf{Case 2 :  $\mathfrak{s}=\mathfrak{so}(3)$.} Algebras with $e=\mu e_1+e_4$, $\mu>0$. Using (\ref{simplifyequa}) we have $[\mathrm{L}(e_4), \mathrm{ad}(e_1)] = 0$ and $\mu\mathrm{R}(e_1)+\mathrm{L}(e_4)=\mathrm{Id}_\h$. Also $\mathrm{tr}(\mathrm{R}(s))=0$  for all $s\in\mathfrak{so}(3)$ . This determines another 25 structure constants by linear equations.  The remaining LSA-structure equations then are almost trivial. It is easy to see that they have a unique solution, which is given by the algebra $\h_4^{\mu>0}$  of Theorem \ref{the 2}.
\end{proof}

\textbf{II. Algebras with central right-identity $e=e_4$.}\\

Let $(\h,\cdot)$ be an LSA-structure on $\G$ ($\G=\mathfrak{sl}_2(\R)\oplus\mathbb{R}e_4$). Since $\mathrm{R}(e_4)=\mathrm{Id}_\h$, $\h$ is completely reducible as $\G$-module. Because $H^0(\G,\h)=0$, we have only two possibilities for $\h$. In the first case, $\h$
is irreducible, and in the second case, $\h=V\oplus V$, where V (as an $\mathfrak{sl}_2(\mathbb{R})$-module) is isomorphic to the $2$-dimensional natural representation of $\mathfrak{sl}_2(\mathbb{R})$.

\begin{lemma}\label{irri2}
	Let $\h=(\mathfrak{sl}_2(\mathbb{R})\oplus\mathbb{R}e_4,\cdot)$ be an LSA-structure on $\mathfrak{sl}_2(\mathbb{R})\oplus\mathbb{R}e_4$. Then, $\mathrm{L}(e_3)$  is similar to $\mathbf{diag}(-3, -1, 1, 3)$ or to $\mathbf{diag}(1,-1, 1, -1)$
	and $\mathrm{L}(e_1)$, $\mathrm{L}(e_2)$  are nilpotent.
\end{lemma}
\begin{proof}
	If $\h$ is irreducible (as an $\mathfrak{sl}_2(\mathbb{R})$-module), then $\h$ has a basis $\{v_0,v_1,v_2,v_3\}$ such that 
	\begin{center}
		$e_3 v_j=(-3+2j)v_j$, $e_1v_j=v_{j+1}$, $e_2v_j=(4-j)v_{j-1}$ and $e_2v_0=0$ for $j=0,...,3$.
	\end{center}
	With respect to this basis, $\mathrm{L}(e_3)=\mathrm{diag}(-3, -1, 1, 3)$ and   $\mathrm{L}(e_1)$,  $\mathrm{L}(e_2)$   are nilpotent. If $\h=V\oplus V$ where $V$ (as an $\mathfrak{sl}_2(\mathbb{R})$-module) is
	isomorphic to the $2$-dimensional natural representation of $\mathfrak{sl}_2(\mathbb{R})$. In this case, we choose a basis according to $V\oplus V$, where $V$ is a highest weight module for $\mathfrak{sl}_2(\mathbb{R})$. Note that the basis $\{v_0,v_1,v_2,v_3\}$ does not satisfy the LSA-structures.
	
\end{proof}

For a Euclidean ring $\mathbb{R}$ it is well know that $\mathbf{GL}_2(\mathbb{R})$, the group of invertible $2\times 2$-matrices over $\mathbb{R}$, is generated by the elementary matrices
\begin{center}
	$\begin{pmatrix}
	1&\beta\\
	0&1
	\end{pmatrix},~\begin{pmatrix}
	\alpha&0\\
	0&\delta
	\end{pmatrix},~\begin{pmatrix}
	0&1\\
	1&0
	\end{pmatrix}$,
\end{center}
where $\alpha, \beta, \delta\in\mathbb{R}$ and $\alpha,\beta$ are units. By explicit calculations now we classify the left-invariant affine structures on $\mathbf{GL}_2(\mathbb{R})$,
i.e., the LSA-structures on $\G=\mathfrak{gl}_2(\mathbb{R})$. Let $$e_1=\begin{pmatrix}
0&1\\
0&0
\end{pmatrix},~e_2=\begin{pmatrix}
0&0\\
1&0
\end{pmatrix},~e_3=\begin{pmatrix}
1&0\\
0&-1
\end{pmatrix},~e_4=\begin{pmatrix}
1&0\\
0&1
\end{pmatrix}$$ be
the canonical $\mathfrak{sl}_2(\mathbb{R})$-basis for $\G$. The Lie algebra automorphisms of $\G$ are 
\begin{center}
	$\psi_A : \G\rightarrow\G,~~X\mapsto AXA^{-1}$ with $A=\begin{pmatrix}
	\alpha&\beta\\
	\gamma&\delta
	\end{pmatrix},$ $\alpha\delta-\gamma\beta\neq 0$
\end{center}
and 
\begin{center}
	$\psi_t : \G\rightarrow\G,~~x\mapsto s+te_4$~ where~  $x=s+e_4$, $s\in\mathfrak{sl}_2(\mathbb{R})$.
\end{center}

\textbf{Lie Algebra:} $\h=(\mathfrak{sl}_2(\mathbb{R})\oplus\mathbb{R}e_4, \cdot)$. Let $\mathrm{L}(e_1)=(x_{k\ell})_{k,\ell}, \mathrm{L}(e_2)=(y_{k\ell})_{k,\ell},\mathrm{L}(e_3)=(z_{k\ell})_{k,\ell}$ with $k,\ell=1,...,4$. Using $\mathrm{L}(x)=\mathrm{R}(x)+\mathrm{ad}(x)$, and the fact that, $[e_j,e_4]=0$ for $j=1,2,3$ and $e_i\cdot e_j-e_j\cdot e_i=[e_i,e_j]$  for all $i,j=1,2,3$, we obtain
\begin{equation}
\mathrm{L}(e_2) : \begin{cases}
e_2\cdot e_1=x_{12}e_1+x_{22}e_2+(x_{32}-1)e_3+x_{42}e_4,\\
e_2\cdot e_2=\sum_j y_{j2}e_j,\\
e_2\cdot e_3=(-x_{12}-y_{22})e_3+\sum_{j\neq 3}y_{j3}e_j,\\
e_2\cdot e_4=e_2,
\end{cases}
\end{equation}
\begin{equation}
\mathrm{L}(e_3):\begin{cases}
e_3\cdot e_1=(x_{13}+2)e_1+x_{23}e_2+(-x_{11}-x_{22})e_3+x_{43}e_4,\\
e_3\cdot e_2=y_{13}e_1+(y_{23}-2)e_2+(-x_{12}-y_{22})e_3+y_{43}e_4,\\
e_3\cdot e_3=z_{13}e_1+z_{23}e_2+(-x_{13}-y_{23})e_3+z_{43}e_4,\\
e_3\cdot e_4=e_3.
\end{cases}
\end{equation}
With $\mathrm{L}(e_4)=\mathrm{Id}_\h$. First we may assume that the upper left-block of the matrix $\mathrm{L}(e_3)$ is equal to zero, that is  $z_{13}=0$ and $z_{23}=0$. By applying $\psi_{\left(\begin{smallmatrix}1&\beta\\0&1\end{smallmatrix}\right)}$ or $\psi_{\left(\begin{smallmatrix}1&0\\\gamma&1\end{smallmatrix}\right)}$. This respects
$\mathrm{L}(e_4)=\mathrm{Id}_\h$ and it is not difficult to see that we can assume $z_{13}=0$ and $z_{23}=0$ or $z_{13}=0$ and $z_{23}=1$. The last case can be reduced to the first  by applying $\psi_{\left(\begin{smallmatrix}1&0\\\gamma&1\end{smallmatrix}\right)}$ and $\psi_{\left(\begin{smallmatrix}1&\beta\\1&\delta\end{smallmatrix}\right)}$.\\
\textbf{Case 1: } $\h=(\mathfrak{sl}_2(\mathbb{R})\oplus\mathbb{R}e_4,\cdot)$ is irreducible. According to the above, we have $z_{13}=0$, $z_{23}=0$, and from the LSA-structures, we extract the following equations:
\begin{eqnarray}
-2\,x_{1 3}^2-x_{1 3}y_{2 3} - x_{2 3}y_{1 3} + z_{4 3}&=&0\\
- 2\,y_{2 3}^2+x_{1 3}y_{2 3}-x_{2 3}y_{1 3}  + z_{4 3}&=&0\\
(x_{1 3} + y_{2 3} - 2)x_{2 3}&=&0\\
(x_{1 3} + y_{2 3} + 2)y_{1 3}&=&0
\end{eqnarray}
From Lemma \ref{irri2}, and the fact that the characteristic polynomial of $\mathrm{L}(e_3)$ is $\mathcal{P}=(\mathcal{X}-1)(\mathcal{X}+1)(\mathcal{X}-3)(\mathcal{X}+3)$, we obtain
\begin{eqnarray}
-z_{{43}} \left( x_{{13}}y_{{23}}-x_{{23}}y_{{13}}-2\,x_{{13}}+2
\,y_{{23}}-4 \right) 
+9&=&0\\
\left( x_{{13}}+y_{{23}}-z_{{43}}+1 \right)\left( x_{{13}}y_{{23}}-x_{{23}}y_{{13}}-3\,x_{{13}}+y_{{23}}-3
\right)  &=&0 \\
\left( x_{{13}}+y_{{23}}+z_{{43}}-1 \right)  \left( x_{{13}}y_{{
		23}}-x_{{23}}y_{{13}}-x_{{13}}+3\,y_{{23}}-3 \right)
&=&0\\
\left( 3\,x_{{13}}+3\,y_{{23}}-z_{{43}}+9 \right)  \left( x_{{13}
}y_{{23}}-x_{{23}}y_{{13}}-5\,x_{{13}}-y_{{23}}+5 \right) 
&=&0\\
\left( 3\,x_{{13}}+3\,y_{{23}}+z_{{43}}-9 \right)\left( x_{{13}}y_{{23}}-x_{{23}}y_{{13}}+x_{{13}}+5\,y_{{23}}+
5 \right) 
&=&0
\end{eqnarray}
It follows that $z_{43} = 1, 3$ or $9$. Then the remaining LSA-equations are very simple. It is easy to see that they have
four solutions, and we already know that, for each nonnegative integer $n$, the group $\mathbf{SL}_2(\mathbb{R})$ has an irreducible representation of dimension $n+1$, which is unique up to an isomorphism. Therefore, we can take the following solution 
\begin{center}
	$ x_{{11}}=0,x_{{12}}=0,x_{{13}}=-1,x_{{21}}=0,x_{{22}}=0,
	x_{{23}}=0,x_{{31}}=\frac{3}{y_{{13}}},x_{{32}}=0,x_{{41}}=\frac{ 3}{y_
		{{13}}}, $ 
\end{center}
\begin{center}
	$x_{{42}}=3,x_{{43}}=0,y_{{12}}=0,,y_{{22}}=0,y_{{23}}=-1,y_{{32}}=-{\frac {y_{{13}}}{4}},y_{{42}}={
		\frac {3\,y_{{13}}}{4}},y_{{43}}=0,z_{{43}}=3$.
\end{center}
We may also normalize $y_{13}$ to 1. That means we may take $y_{13}=1$. This solution is  given by the algebra $\h_3$.\\
\textbf{Case 2: } $\h=(\mathfrak{sl}_2(\mathbb{R})\oplus\mathbb{R}e_4,\cdot)$ is reducible, i.e., $\h=V\oplus V$. The characteristic polynomial is now $\mathcal{P}_c=(\mathcal{X}-1)^2(\mathcal{X}+1)^2$, and its minimal polynomial is $\mathcal{P}_{min}=(\mathcal{X}^2-1)$. Then $\mathcal{P}_{min}(\mathrm{L}(e_3))=0$, this implies  $z_{13}=0,z_{23}=0,z_{43}=1$, and this determines another 11 structure constants by linear equations. The
remaining LSA-structure equations then are almost trivial. After a short calculation we obtain a contradiction.

\begin{lemma}\label{irri1}
	Let $\h=(\mathfrak{so}(3)\oplus\mathbb{R}e_4,\cdot)$ be an LSA-structure on $\mathfrak{so}(3)\oplus\mathbb{R}e_4$. Then, $\h$ has a basis $\{v_0,v_1,v_2,v_3\}$ such that $\mathrm{L}(e_1)$, $\mathrm{L}(e_2)$ and $\mathrm{L}(e_3)$ have the following expressions$:$
	\begin{equation}
	\mathrm{L}(e_1)=\frac{1}{2}(-v_3,-v_2,v_1,v_0),~~\mathrm{L}(e_2)=\frac{1}{2}(-v_1,v_0,-v_3,v_2),~~\mathrm{L}(e_3)=\frac{1}{2}(-v_2,v_3,v_0,-v_1).
	\end{equation}
\end{lemma} 
\begin{proof}
Semisimple Lie algebra $\mathfrak{so}(3)$ has a minimal representation of dimension  $3$, then  one possibility for $\h$  is that it is  irreducible (as an $\G$-module). Therefore, $\h$ has a basis $\{v_0,v_1,v_2,v_3\}$ such that $\mathrm{L}(e_1)$, $\mathrm{L}(e_2)$ and $\mathrm{L}(e_3)$ are described as Lemma \ref{irri1}. Note that this basis $\{v_0,v_1,v_2,v_3\}$ does not satisfy the LSA-condition. 
\end{proof}

\textbf{Lie Algebra:} $\h=(\mathfrak{so}(3)\oplus\mathbb{R}e_4, \cdot)$. Let $\mathrm{L}(e_1)=(x_{k\ell})_{k,\ell}, \mathrm{L}(e_2)=(y_{k\ell})_{k,\ell},\mathrm{L}(e_3)=(z_{k\ell})_{k,\ell}$ with $k,\ell=1,...,4$. Using $\mathrm{L}(x)=\mathrm{R}(x)+\mathrm{ad}(x)$, and the fact that, $[e_j,e_4]=0$ for $j=1,2,3$ and $e_i\cdot e_j-e_j\cdot e_i=[e_i,e_j]$  for all $i,j=1,2,3$, we obtain
\begin{equation}
\mathrm{L}(e_2) : \begin{cases}
e_2\cdot e_1=x_{12}e_1+x_{22}e_2+(x_{32}-1)e_3+x_{42}e_4,\\
e_2\cdot e_2=\sum_j y_{j2}e_j,\\
e_2\cdot e_3=\sum_{j}y_{j3}e_j,\\
e_2\cdot e_4=e_2,
\end{cases}
\end{equation}
\begin{equation}
\mathrm{L}(e_3):\begin{cases}
e_3\cdot e_1=x_{13}e_1+(1+x_{23})e_2+x_{33}e_3+x_{43}e_4,\\
e_3\cdot e_2=(y_{13}-1)e_1+\sum_{j\neq 1}y_{j3}e_j,\\
e_3\cdot e_3=\sum_{j}z_{j3}e_j,\\
e_3\cdot e_4=e_3.
\end{cases}
\end{equation}
With $\mathrm{L}(e_4)=\mathrm{Id}_\h$. Since the trace of $\mathrm{L}(e_1)$, $\mathrm{L}(e_2)$, $\mathrm{L}(e_3)$ is zero, we have $z_{33}=-x_{13}-y_{23}$, $y_{33}=-x_{12}-y_{22}$, $x_{33}=-x_{11}-x_{22}$. \\
According to Lemma \ref{irri1}, we have only one possibility for $\h$, that is $\h$ is irreducible. From the fact that the characteristic polynomial is independent from the choice of basis, we have $\mathcal{P}_c=(\mathcal{X}^2+\frac{1}{4})^2$  as a characteristic polynomial of $\mathrm{L}(e_3)$, and its minimal polynomial is $\mathcal{P}_{min}=(\mathcal{X}^2+\frac{1}{4})$, we have also $\mathrm{L}^2(e_3)+\frac{1}{4}\mathbf{I}=0$, from this one, we obtain $z_{13}=0, z_{23}=0, z_{33}=0$ and $z_{43}=-\frac{1}{4}$. This determines another 8 structure constants by linear equations. The same  we simplify $\mathrm{L}(e_1)$ and $\mathrm{L}(e_2)$. Then the remaining LSA-equations are almost trivial. It is easy to see that they have a unique solution, which is given by the algebra $\h_5$.
\\\\
\textbf{III. Algebras with  right-identity $e=e_1+\nu e_2+e_4$.}
\\\\
Let $\h$ be a Lie algebra and $\rho : \h\rightarrow\mathfrak{gl}(V)$ be a representation of $\h$.
A 1-cocycle $\varphi$ associated to $\rho$ is defined as a linear map from $\h$ to $V$ satisfying
\begin{equation}\label{1cocycle}
\varphi\big([x,y]\big)=\rho(x)\varphi(y)-\rho(y)\varphi(x),~~\text{for all }x,y\in\h.
\end{equation}
If, in addition, $\varphi$ is a linear isomorphism (thus $\dim V =
\dim \h)$, $(\rho,\varphi)$ is said to be bijective.

Let $(\rho,\varphi)$ be a bijective 1-cocycle, then it is easy to see that
\begin{equation}\label{newproduct}
x\ast y=\varphi^{-1}\big(\rho(x)\varphi(y)\big),~~\text{for all }  x,y\in\h,
\end{equation}
defines a left-symmetric algebra on $\h$ (Medina, \cite{Medina}). Conversely, for a left-symmetric algebra $(\rho,\mathrm{Id}_\h)$ is a bijective 1-cocycle of $\h$. There is a bijection between the set of the
isomorphism classes of bijective 1-cocycles of $\h$ and the set of left-symmetric algebras
on $\h$. Under this correspondence equivalent bijective 1-cocycles are mapped to isomorphic left-symmetric algebras, and vice versa \cite{Bai}. 

Note that in our case, we also have 
\begin{equation}\label{neutreele}
\varphi(x)=\rho(x)\varphi(e),~~ \text{for all } x\in\h,
\end{equation}
where $e$ is the unique right-identity of $\h$.
\\
Then a procedure is provided to classify left-symmetric algebras in terms of classification of equivalent classes of bijective
1-cocycles. Let $\h$ be a given Lie algebra with a basis $\{e_1,e_2,...,e_n\}$. For a representation $\rho : \h\rightarrow\mathfrak{gl}(V)$ ($\dim\h=\dim V$) with a basis $\{v_0,...,v_{n-1}\}$ of $V$, we can let $\rho(x)=\big(\rho_{ij}(x)\big)$ for any $x \in\h$, where $\rho_{ij} : \h\rightarrow\mathbb{R}$ be linear functions. On the other hand, let $\varphi :\h\rightarrow V$ be a 1-cocycle, then we can let $\varphi(x)=\sum^{n-1}_k\varphi_k(x)v_k$, where $\varphi_k : \h\rightarrow\mathbb{R}$ are linear functions. The conditions
of the representation $\rho$ and the 1-cocycle $\varphi$ can give a series of equations for linear
functions $\rho_{ij}$ and $\varphi_k$. For more details about this procedure see \cite{Bai}.

\textbf{Case 1:} If $\h=\mathfrak{sl}_2(\R)\oplus \R e_4=V\oplus V$. Let $e=e_1+\nu e_2+e_4$, where $\nu<0$. From Lemma \ref{irri2} and using $(\ref{1cocycle})$, $(\ref{neutreele})$. It is easy to see that $\varphi$ is given by
\begin{equation}
\varphi=\left( \begin {array}{cccc} a\nu+b&0&a&-a\nu+a-b
\\ \noalign{\medskip}0&a&-a\nu-b&b\\ \noalign{\medskip}\nu\,c+d&0
&c&-\nu\,c+c-d\\ \noalign{\medskip}0&c&-\nu\,c-d&d\end {array}
\right).
\end{equation}
$\varphi$ is bijective if and only if  $(ad - bc)\neq 0$. There  is no doubt that $(\ref{newproduct})$ defined  a unique left-symmetric product on $\h$, which is given by the algebra $\h^{\nu<0}_4$. As you can see, the condition $\mathrm{R}(e)=\mathrm{Id}_\h$, determines the expression of $\mathrm{L}(e_4)$.\\\\
\textbf{Case 2:} If $\h$ is irreducible (as an $\mathfrak{sl}_2(\mathbb{R})$-module), then $\h$ has a basis $\{v_0,v_1,v_2,v_3\}$ such that 
\begin{center}
	$\rho( e_3) v_j=(-3+2j)v_j$, $\rho(e_1)v_j=v_{j+1}$, $\rho(e_2)v_j=(4-j)v_{j-1}$ and $\rho(e_2)v_0=0$, 
\end{center} 
for $j=0,...,3$.

According to $(\ref{1cocycle})$, the bijective 1-cocycles are as follows:
\begin{equation}
\varphi=\left( \begin {array}{cccc} 0&0&-3\,u&a\\ \noalign{\medskip}u&2\,v&0&
b\\ \noalign{\medskip}0&0&v&c\\ \noalign{\medskip}v&0&0&d\end {array}
\right), ~~\text{with~~} v(av + 3cu)\neq 0.
\end{equation}
Let $e=e_1+\nu e_2+e_4$. The condition $(\ref{neutreele})$ leads to $\mathrm{det}(\varphi)=0$, as it is obvious. Therefore, and in this case $\h=\mathfrak{sl}_2(\mathbb{R})\oplus\mathbb{R}e_4$ does not admit a left-symmetric product.

\begin{proposition} \label{pr7}
Let $\G$ be an eight-dimensional Frobeniusian  non-solvable  Lie algebra. Then there exists
a basis $\{f_1,\ldots,f_8\}$ of $\G$ such that 
\begin{equation}
\om_0=f^{15}+f^{26}+f^{37}+f^{48}.
\end{equation}
and the non vanishing Lie brackets as listed in Table $\ref{tab3}$.
\end{proposition}
\begin{proof} Let $\h$ be one of the left-symmetric algebras listed in
	Table \ref{tab2}.  From Proposition \ref{pr 3}, every
	eight-dimensional Frobeniusian  non-solvable  Lie algebras have the	form $(\G_{\rho,\al}=\h\oplus \h^*,\om_0)$ given by non-zero Lie	brackets  defined by $(\ref{bra})$, with $\al\in Z^2_\rho(\h,\h^*)$
	satisfies $(\ref{sym})$
	and  $\om_0$ is defined by the dual pairing of $\h$ and $\h^*$.
	It is straightforward to verify that, $\al\in Z^2_\rho\big(\h,\h^\ast\big)$, i.e.,
	\begin{equation}
	\big(\partial_\rho\alpha\big)(x,y,z)=\oint_{x,y,z}[[x,y],z]=0,~\text{for all}~x,y,z\in\h.
	\end{equation}
 We continue the proof in the case $\h^{\lambda>0}_2$, the other cases are treated in the Appendix. Let $\G_{\rho,\al}^2=\h^{\lambda>0}_2\oplus (\h^{\lambda>0}_2)^\ast$ and
	$\{f_1,\ldots,f_8\}$ be a basis of $\G_{\rho,\al}^2$, with
	$\h^{\lambda>0}_2=\langle  f_1,f_2,f_3,f_4\rangle$,
	$(\h^{\lambda>0}_2)^*=\langle f_5,f_6,f_7,f_8\rangle$ and
	$\al(f_i,f_j)=\sum_{k=5}^{8}\al_{ij}^kf_k$ for $i,j\in\{1,\ldots,4\}$. Let $\phi\in\mathcal{C}^1\big(\h_2^{\lambda>0},(\h_2^{\lambda>0})^\ast\big)$ defined by
	\begin{eqnarray*}
\phi(f_1)&=&\frac{\al_{24}^5}{3} f_5+\frac{\al_{24}^7-2\,\alpha_{14}^5}{6}f_7+\left(\frac{\al_{14}^5(2\,\lambda+6)}{6}-\frac{\al_{24}^7(\lambda-3)}{6}\right)f_8,\\
\phi(f_2)&=&-\frac{\al_{34}^6}{3} f_6-\frac{\alpha_{34}^7+2\,\al_{14}^6}{6}f_7+\left(\frac{\al_{14}^6(2\,\lambda-6)}{6}+\frac{\al_{34}^7(\lambda+3)}{6}\right)f_8,\\
\phi(f_3)&=&  \frac{\al_{24}^7-2\,\alpha_{14}^5}{6}f_5-\frac{\alpha_{34}^7+2\,\al_{14}^6}{6}f_6+(\alpha_{34}^5-\alpha_{24}^6)f_7+\left(\alpha_{24}^6(-1 + \lambda) - \alpha_{34}^5(1 + \lambda)\right)f_8,\\
\phi(f_4)&=&  \left(\frac{\al_{14}^5(2\,\lambda+6)}{6}-\frac{\al_{24}^7(\lambda-3)}{6}\right)f_5+\left(\frac{\al_{14}^6(2\,\lambda-6)}{6}+\frac{\al_{34}^7(\lambda+3)}{6}\right)f_6\\
 &&+\left(\alpha_{24}^6(-1 + \lambda) - \alpha_{34}^5(1 + \lambda)\right)f_7+\left(-\alpha_{24}^6(-1 + \lambda)^2 + \alpha_{34}^5(-1+\lambda^2) - 2\,\alpha_{13}^5\right)f_8.
	\end{eqnarray*}
It is easy to show that, $\partial_\rho\phi=\alpha$, for all $\alpha\in Z^2_\rho\big(\h_2^{\lambda>0},(\h_2^{\lambda>0})^\ast\big)\bigcap\mathcal{C}^2_L(\h_2^{\lambda>0},(\h_2^{\lambda>0})^\ast)$ and this implies $H^2_\rho\big(\h_2^{\lambda>0},(\h_2^{\lambda>0})^\ast\big)=0$. Note that for all Lie algebras $\h_j$ given in Table \ref{tab3}, we have	 $H^2_\rho\big(\h_j,(\h_j)^\ast\big)=0$. It is simply a matter of noting
	$(\G^2_{\rho,\alpha},\omega_0)=(\G^2_{\rho},\omega_0)$. By identifying
	the basis $\langle f_5,f_6,f_7,f_8\rangle$ with the dual basis of
	$\langle f_1,f_2,f_3,f_4\rangle$, $\om_0$ becomes
	\begin{equation}
	\om_0=f^{15}+f^{26}+f^{37}+f^{48},
	\end{equation}
	and the non-vanishing Lie brackets are given in Table $\ref{tab3}$.
	
\end{proof}

\begin{proposition}\label{pr I}
	Let $(\G,\om)$ be an eight-dimensional Frobeniusian  indecomposable non-solvable  Lie algebra. Then
$(\G,\om)$ is symplectomorphically equivalent to one of the following Lie algebras equipped
	with a symplectic form as follows:
	
{\renewcommand*{\arraystretch}{1.3}
	\begin{tabular}{ll}
$\mL_{8,3} :$&$ \omega=e^{17}+e^{25}+e^{36}-2e^{48}.$
\\		
	
	$\mL_{8,4}^{p\neq0} :$&$ \omega=e^{14}-e^{15}+e^{26}+e^{27}-e^{34}-e^{35}+2e^{48}-2e^{58}-2pe^{68}+2pe^{78}.$\\
	$\mL_{8,16} :$&$ \omega=e^{15}-e^{16}-e^{27}-e^{34}-e^{58}-e^{68}-e^{78}.$\\
	$\mL_{8,17}^{p\in]-1,1]\setminus\{0\}} :$&$\omega=-e^{14}-e^{17}-e^{25}+e^{36}-e^{48}+pe^{78}.$\\		
	$\mL_{8,17}^{p=0}(a_{67}=0) :$&$ \omega=-e^{14}-e^{17}-e^{25}+e^{36}-e^{48}.$\\
	$\mL_{8,18}^{p>0} :$&$\omega=pe^{15}-e^{16}-e^{27}-pe^{34}+e^{48}-p^2e^{58}-p(e^{68}+e^{78}).$\\
	$\mL_{8,20}:$&$ \omega=-e^{15}+\frac{1}{2} e^{17}-2e^{26}-e^{34}+\frac{1}{2} e^{36}-e^{58}-\frac{1}{6} e^{78}.$	
\end{tabular}}	
\end{proposition}
\begin{proof}
	
	The Proposition \ref{pr7} confirms that for each Lie algebra $\G_\rho^i$ of  Table \ref{tab3} there  exist a basis	$\{f_1,\ldots,f_8\}$ such that the symplectic structure is given by
	\[\om_0=f^{15}+f^{26}+f^{37}+f^{48},\]
	and the non vanishing Lie brackets are given in Table $\ref{tab3}$. Then we construct a family of isomorphisms from $\G_\rho^i$ to the eight-dimensional Frobeniusian  non-solvable  Lie algebras $\mL_{8,j}$ given by Turkowski \cite{T}. 
	Each isomorphism  $\Psi$ (see, Table \ref{tab4}) is given  from $\{f_1,\ldots,f_8\}$ to $\{e_1,\ldots,e_8\}$. The new symplectic structures $\om_j$ in $\mL_{8,j}$ are given  by $\om_j=\Psi_*(\omega_0)$.
	
\textit{Uniqueness of the symplectic structures $\omega_j$.}	 From Proposition \ref{pr2}, the extension triple $(\h,\cdot, [\alpha])$, where $[\alpha]\in H^2_{L,\rho}(\h,\h^\ast)$ induces a bijection between isomorphism classes of symplectic Lie algebras with Lagrangian ideal and isomorphism classes of left-symmetric Lie algebras with symplectic extension cohomology class. 

It is easy to see that, for the case given in Proposition \ref{pr7},  $\phi\in C^1_L\big(\h^{\lambda>0}_2,(\h^{\lambda>0}_2)^\ast\big)$, that is, 
\begin{center}
$\phi(x)(y)-\phi(y)(x)=0$, for all $x,y\in\h_2^{\lambda>0}$. 
\end{center}
Therefore, $\al=\partial_\rho\phi\in
\partial_\rho C^1_L\big(\h^{\lambda>0}_2,(\h^{\lambda>0}_2)^\ast\big)=B^2_{L,\rho}(\h^{\lambda>0}_2,(\h^{\lambda>0}_2)^\ast\big)$, and this implies that $$H^2_{L,\rho}(\h^{\lambda>0}_2,(\h^{\lambda>0}_2)^\ast\big)=0.$$
Hence the  uniqueness. Clearly, we have $H^2_{L,\rho}(\h_j,(\h_j)^\ast\big)=0$ for all Lie algebras listed in Table \ref{tab3}. Which completes the proof.
\end{proof}
\begin{remark}
About the uniqueness of symplectic structures. Note that, the symplectic form in general position given in Proposition $\ref{pr5}$ (for all Frobeniusian Lie algebras) admits a unique Lagrangian ideal. Therefore, all symplectic structures in an eight-dimensional Frobeniusian  indecomposable non-solvable  Lie algebra admit a Lagrangian ideal. A further feature of these structures is that they are symplectically isomorphic.
\end{remark}

\subsection{Type II}
Now consider the symplectic Lie algebras $\mL^{p,-p\neq0}_{8,7},\;\mL^0_{8,8},\;\mL^{0,q\neq0}_{8,9}$ and $\mL^0_{8,17}$ for $a_{67}\not=0$, denote by $(\G,\om)$ one of these algebras. By Corollary \ref{co2}, $\G$ admit $\mathfrak{a}=\langle e_6,e_7\rangle$ as a symplectic ideal then $\G$ is a semi-direct product
\[\G=\mathfrak{a}^{\perp}\ltimes\mathfrak{a}.\]
On the one hand, the pair $(\mathfrak{a}^{\perp},\om_{|\mathfrak{a}^{\perp}})$ is a six-dimensional  symplectic non-solvable  Lie algebra so $\mathfrak{a}^{\perp}$ is isomorphic to the Lie algebra of affine transformations  $\aff$. On the other hand, the non-vanishing Lie brackets in 
$\mathfrak{a}^{\perp}=\langle e_1,\ldots,e_5,e_8\rangle$ are given by
\begin{eqnarray*}
&&[e_1,e_2]=2e_2,\, [e_1,e_3]=-2e_3,\, [e_1,e_4]=e_4,\, [e_1, e_5]=-e_5,\,[e_2,e_3]=e_1\\&&[e_2,e_5] = e_4,\, [e_3,e_4] = e_5,\, [e_4,e_8] = e_4,\,[e_5,e_8] = e_5.
\end{eqnarray*}
The affine Lie algebra $\aff$ admits symplectic structures and all of them are symplectically isomorphic. Therefore, there are automorphisms of $\mathfrak{a}^{\perp}$ such that the restriction of the symplectic form is given by
\[\om_{|\mathfrak{a}^{\perp}}=e^{12}+e^{15}-e^{34}-e^{58}.\]
\begin{proposition}\label{pr II} Let $(\G,\om)$ be one of the following symplectic Lie algebras  $\mL^{p,-p\neq0}_{8,7},\;\mL^{0}_{8,8},\;\mL^{0,q\neq0}_{8,9}$ and $\mL^0_{8,17}$ for $a_{67}\not=0$. Then $\om$  is symplectomophic to
	\[\om=e^{12}+e^{15}-e^{34}-e^{58}\mp e^{67}.\]	
	\end{proposition}
	\begin{proof}
Based on what has preceded, it is sufficient to compute all automorphisms of $\mathfrak{a}^{\perp}$ in order to have automorphisms of $\G=\mathfrak{a}^{\perp}\ltimes\mathfrak{a}$. We get that, any automorphism of $\G$ is of the following form
\begin{equation}
\Psi:=\left(
\begin{array}{c|cccc}
\phi&&0& \\\hline
\ast&x&&\ast&\\
\ast&\ast&&x&
\end{array}
\right)\in\mathrm{Aut}(\mathfrak{a}^{\perp}\ltimes\mathfrak{a}),
\end{equation}
where $\phi\in\mathrm{Aut}(\aff)$ and $x\in\R^\ast$. On the other hand, we have
\begin{center}
$\omega_\lambda=e^{12}+e^{15}-e^{34}-e^{58}+\lambda e^{67}$.
\end{center}
In conclusion, the symplectic structures $\om_j$ in $\G=\mathfrak{a}^{\perp}\ltimes\mathfrak{a}$ are given by $\omega_j=\Psi^\ast\big(\omega_\lambda\big)$.

We give the detailed proof for  Lie algebra $\G=\mL_{8,8}^0$, similar treatment is given to all remaining cases. 

Let $\G=\mathfrak{a}^{\perp}\ltimes\mathfrak{a}$, To simplify, let us note  $\mathfrak{a}^{\perp}=\mathrm{span}(e_1,\ldots,e_6)$ and $\mathfrak{a}=\mathrm{span}(e_7,e_8)$. The symplectic structures becomes
\begin{align}\label{wl88}
\om&=a_{12}e^{12}+a_{13}e^{13}+a_{23}e^{23}+a_{78}e^{78}+a_{68}e^{68}+a_{78}e^{78}&\nonumber\\&+a_{25}(e^{14}+e^{25}+e^{46})
+a_{34}(-e^{15}+e^{56}+e^{34}), 
\end{align} 
with $a_{7 8}(a_{1 2}a^2_{34} + a_{1 3}a^2_{25} + 2a_{2 3}a_{25}a_{34})\neq 0$, and 
\[\om_{|\mathfrak{a}^{\perp}}=a_{12}e^{12}+a_{13}e^{13}+a_{23}e^{23}+a_{25}(e^{14}+e^{25}+e^{46})+a_{34}(-e^{15}+e^{34}+e^{56}),\]
with $(a_{12}a_{34}^2+a_{13}a_{25}^2 + 2a_{23}a_{25}a_{3 4})\neq 0$. Let $\Psi_1,\Psi_2$ and $\Psi_3\in\mathrm{Aut}(\mathfrak{a}^{\perp}\ltimes\mathfrak{a})$ given as follows
\begin{center}
$\Psi_j:=\left(
   \begin{array}{c|cccc}
\phi_j&&0& \\\hline
      \ast&x&&\ast&\\
      \ast&\ast&&x&
   \end{array}
\right),$
\end{center}
where, $x\in\R^\ast$, $\phi_j\in\mathrm{Aut}(\mathfrak{a}^{\perp})$ and $\Psi_j$ are given below. This form can be found by applying the definition of Lie algebra automorphisms. On the other hand, one applies the definition making use of the automorphisms
of symplectic Lie algebras (see below for instance). The determinant of $\omega_{|\mathfrak{a}^{\perp}}$ is noted by  $\delta$, and we have 
\\\\
If $a_{25}\neq 0$, then
\begin{center}
$\omega_0:=\phi^\ast_1(\om_{|\mathfrak{a}^{\perp}})=e^{12}+e^{15}-e^{34}-e^{56}$,
\end{center}
where,
\begin{center}
$\phi_1:=\left(\begin {array}{cccccc} -1&{\frac {a_{{34}}a_{{25}}}{{\delta}}}&0&0&0&0\\ \noalign{\medskip}-2\,{\frac {a_{{34}}}{a_{{25}}}}
&{\frac {{a_{{34}}}^{2}}{{\delta}}}&-{\frac {{\delta}}{{a_{{25}}^2}
}}&0&0&0\\ \noalign{\medskip}0&-{\frac {{a_{{25}}^2}}{{\delta
}}}&0&0&0&0\\ \noalign{\medskip}{\frac {3\,a_{{12}}a_{{34}}+2\,a_{{2
3}}a_{{25}}}{2\,{a_{{25}}^2}}}&-{\frac {a_{{34}} \left( a_{{12}
}a_{{34}}+a_{{23}}a_{{25}} \right) }{a_{{25}}{\delta}}}&{\frac {
{\delta}\,a_{{12}}}{2\,{a_{{25}}^3}}}&-{\frac {a_{{34}}}{{\delta}}}&-\frac{1}{a_{{25}}}&{\frac {a_{{12}}a_{{34}}+2\,a_{{23}}a_{{25}}}{2\,{a_{{25}}^2}}}\\ \noalign{\medskip}-{\frac {a_{{12}}}{2\,
a_{{25}}}}&{\frac {a_{{12}}a_{{34}}+a_{{23}}a_{{25}}}{{\delta}}
}&0&{\frac {a_{{25}}}{{\delta}}}&0&{\frac {a_{{12}}}{2\,a_{{25}}}
}\\ \noalign{\medskip}0&0&0&0&0&1\end {array} \right). 
$
\end{center}
If $a_{25}=0$, this implies that $a_{12}a_{34}\neq 0$, and we have
\begin{center}
$\omega_1:=\phi^\ast_2(\om_{|\mathfrak{a}^{\perp}})=e^{13}+e^{14}+e^{25}+e^{46}$,
\end{center}
where,
\begin{center}
$\phi_2:=\left( \begin {array}{cccccc} -1&0&0&0&0&0\\ \noalign{\medskip}0&0&-\frac{1}{
a_{{12}}}&0&0&0\\ \noalign{\medskip}0&-a_{{12}}&0&0&0&0
\\ \noalign{\medskip}-{\frac {a_{{13}}}{2\,a_{{34}}}}&0&-{\frac {a_{{23}}}{a_{{12}}a_{{34}}}}&0&-{\frac {1}{a_{{12}}a_{{34}}}}&-{
\frac {a_{{13}}}{2\,a_{{34}}}}\\ \noalign{\medskip}{\frac {a_{{23}}
}{a_{{34}}}}&-{\frac {a_{{12}}a_{{13}}}{2\,a_{{34}}}}&0&\frac{1}{a_{{34}}
}&0&-{\frac {a_{{23}}}{a_{{34}}}}\\ \noalign{\medskip}0&0&0&0&0
&1\end {array} \right)$.
\end{center}
In addition, we have
\begin{center}
$\phi^\ast_3(\om_1)=\omega_0$,
\end{center}
where,
\begin{center}
$ \phi_3:=\left( \begin {array}{cccccc} -1&0&0&0&0&0\\ \noalign{\medskip}0&0&-1
&0&0&0\\ \noalign{\medskip}0&-1&0&0&0&0\\ \noalign{\medskip}0&0&0&0&-1
&0\\ \noalign{\medskip}0&0&0&1&0&0\\ \noalign{\medskip}0&0&0&0&0&1
\end {array} \right)$.
\end{center}
As a result, we obtain by completing the automorphisms $\phi_1,\phi_2,\phi_3\in\mathrm{Aut}(\mathfrak{a}^{\perp})$ for will be  automorphisms $\Psi_1,\Psi_2,\Psi_3\in\mathrm{Aut}(\mathfrak{a}^{\perp}\ltimes\mathfrak{a})$, i.e., 
\begin{center}
$\Psi_j:=\left(
   \begin{array}{cccc|cccccc}
& \phi_j&&&0&&0 \\\hline
     0 &\ldots&0&c&a&&0&\\
     0 &\ldots&0&d&b&&a
   \end{array}
\right)$,
\end{center}
with $a\in\R^\ast$. Consider the general form of $\mL_{8,8}^0$ given by $(\ref{wl88})$, and let $c=-\frac{  a_{6 8}}{a_{7 8}}$ and $d=\frac{a_{6 7}}{a_{7 8}}$, we have \begin{eqnarray*}
\Omega_0&:=&\Psi^\ast_1(\omega)=\omega_0+a^2\,a_{78}e^{78},\\
\Omega_1&:=&\Psi^\ast_2(\omega)=
\omega_1+a^2\,a_{78}e^{78},\\
\Omega_0&:=&\Psi^\ast_3(\Omega_1),
\end{eqnarray*}
for all $a\in\R^\ast$. The result is obtained by taking $a=\frac{1}{\sqrt{|a_{78}|}}$. Similar computations on each symplectic Lie algebra complete
the proof. 
\end{proof}

\section{Appendix} 
This appendix contains Table \ref{pr2} which corresponds to Theorem \ref{the 2}, which shows the left-symmetric product of $\so(3)\oplus\R e_4$ and $\mathfrak{sl}_2(\R)\oplus \R e_4$.  Table \ref{tab3} presents the new Lie brackets of Frobeniusian non-solvable Lie algebras. For each Lie algebra $\G_\rho$ in Tables \ref{tab3}, we give all the isomorphisms from the $\G_\rho$ to Turkowski algebras. Finally, we give the rest of the proof of  Proposition \ref{pr7}.
\subsection{Left-symmetric product in $\so(3)\oplus\R e_4$ and $\mathfrak{sl}_2(\R)\oplus \R e_4$}
{\renewcommand*{\arraystretch}{1.3}
	\begin{longtable}{|c|l|}
		\hline
		Algebra with  $\mathfrak{s}=\msl$ &  Left-symmetric product
		\\			\hline
		\multirow{5}{*}{$\mathcal{\h}_1$}&$e_1\cdot e_2=\frac{1}{2}(e_1+e_3+e_4), e_1\cdot e_3=-e_1,\, e_1\cdot e_4=e_1,$\\&$e_2\cdot e_1=\frac{1}{2}(e_1-e_3+e_4), e_2\cdot e_3=e_2, e_2\cdot e_4=\frac{1}{2}(-e_1+2e_2+e_3-e_4),$\\
		&$e_3\cdot e_1=e_1, e_3\cdot e_2=-e_2, e_3\cdot e_3=e_1+e_4, e_3\cdot e_4=-e_1+e_3$,\\
		&$e_4\cdot e_1=e_1, e_4\cdot e_2=\frac{1}{2}(-e_1+2e_2+e_3-e_4), e_4\cdot e_3=-e_1+e_3,$ \\
		&$e_4\cdot e_4=-e_1+e_4$.\\
		\hline
		\multirow{5}{*}{$\h^{\lambda>0}_2$}&$e_1\cdot e_2=\frac{1+\lambda}{2}e_3+\frac{1}{2}e_4, e_1\cdot e_3=-e_1,\,\,e_1\cdot e_4=(1+\lambda)e_1,$\\
		&$e_2\cdot e_1=\frac{\lambda-1}{2}e_3+\frac{1}{2}e_4, e_2\cdot e_3=e_2, e_2\cdot e_4=(1-\lambda)e_2,$\\
		&$e_3\cdot e_1=e_1, e_3\cdot e_2=-e_2, e_3\cdot e_3=\lambda e_3+e_4, e_3\cdot e_4=(1-\lambda^2)e_3-\lambda e_4$,\\
		&$e_4\cdot e_1=(1+\lambda)e_1, e_4\cdot e_2=(1-\lambda)e_2, e_4\cdot e_3=(1-\lambda^2)e_3-\lambda e_4,$\\
		&$ e_4\cdot e_4=\lambda(\lambda^2-1)e_3+(1+\lambda^2)e_4$.\\\hline
		\multirow{5}{*}{$\h_3$}&$e_1\cdot e_1= 3(e_3+e_4), e_1\cdot e_2=3e_4, e_1\cdot e_3=-e_1,$\\
		&$e_2\cdot e_1=-e_3+3e_4, e_2\cdot e_2=-\frac{1}{4}e_3+\frac{3}{4}e_4, e_2\cdot e_3=e_1-e_2,$\\
		&$e_3\cdot e_1=e_1, e_3\cdot e_2=e_1-3e_2, e_3\cdot e_3=2e_3+3e_4.$\\
		&$e_j\cdot e_4= e_4\cdot e_j=e_j,~ j=1,\ldots,4.$
		\\\hline
		
		\multirow{5}{*}{$\h_4^{\nu<0}$}&$e_1\cdot e_2=\frac{1}{2}(e_1+\nu e_2+e_3+e_4), e_1\cdot e_3=-e_1$, \\
		&$e_1\cdot e_4=(-\frac{\nu}{2}+1)e_1-\frac{\nu^2}{2}e_2-\frac{\nu}{2}(e_3+e_4),$\\
		&$e_2\cdot e_1=\frac{1}{2}(e_1+\nu e_2-e_3+e_4), e_2\cdot e_3=e_2,$\\ &$e_2\cdot e_4=-\frac{1}{2}e_1+(-\frac{\nu}{2}+1)e_2+\frac{1}{2}(e_3-e_4),$\\
		&$e_3\cdot e_1=e_1, e_3\cdot e_2=-e_2, e_3\cdot e_3=e_1+\nu e_2+e_4$\\
		&$e_3\cdot e_4=-e_1+\nu e_2+e_3,$\\
		&$e_4\cdot e_1=(-\frac{\nu}{2}+1)e_1-\frac{\nu^2}{2}e_2-\frac{\nu}{2}(e_3+e_4),$\\
		&$e_4\cdot e_2=-\frac{1}{2}e_1+(-\frac{\nu}{2}+1)e_2+\frac{1}{2}(e_3-e_4),$\\
		&$e_4\cdot e_3=-e_1+\nu e_2+e_3, e_4\cdot e_4=(\nu-1)e_1+\nu(\nu-1)e_2+(\nu+1)e_4$.
		\\\hline\hline
						Algebra with $\mathfrak{s}=\mathfrak{so}(3)$& Left-symmetric product
				\\			\hline
				\multirow{5}{*}{$\h_5^{\mu>0}$}&$e_1\cdot e_1=-\frac{\mu}{4}e_1-\frac{1}{4}e_4, e_1\cdot e_2=\frac{1}{2}e_3, e_1\cdot e_3=-\frac{1}{2}e_2, e_1\cdot e_4=(\frac{\mu^2}{4}+1)e_1+\frac{\mu}{4}e_4,$\\
				&$e_2\cdot e_1=-\frac{1}{2}e_3,  e_2\cdot e_2=-\frac{\mu}{4}e_1-\frac{1}{4}e_4, e_2\cdot e_3=\frac{1}{2}e_1, e_2\cdot e_4=e_2+\frac{\mu}{2}e_3,$\\
				&$e_3\cdot e_1=\frac{1}{2}e_2, e_3\cdot e_2=-\frac{1}{2}e_1, e_3\cdot e_3=-\frac{\mu}{4}e_1-\frac{1}{4}e_4, e_3\cdot e_4=-\frac{\mu}{2}e_2+e_3, $\\
				&$e_4\cdot e_1=(\frac{\mu^2}{4}+1)e_1+\frac{\mu}{4}e_4, e_4\cdot e_2=e_2+\frac{\mu}{2}e_3,  e_4\cdot e_3=-\frac{\mu}{2}e_2+e_3$, \\
				&$e_4\cdot e_4=-\frac{\mu(\mu^2+4)}{4}e_1+(-\frac{\mu^2}{4}+1)e_4$.\\\hline
				\multirow{5}{*}{$\h_6$}&$e_1\cdot e_1=-\frac{1}{4}e_4, e_1\cdot e_2=\frac{1}{2}e_3,\,\,e_1\cdot e_3=-\frac{1}{2}e_2,$\\
				&$e_2\cdot e_1=-\frac{1}{2}e_3, e_2\cdot e_2=-\frac{1}{4}e_4, e_2\cdot e_3=\frac{1}{2}e_1,$\\
				&$e_3\cdot e_1=\frac{1}{2}e_2, e_3\cdot e_2=-\frac{1}{2}e_1, e_3\cdot e_3=-\frac{1}{4}e_4,$\\
				&$e_j\cdot e_4= e_4\cdot e_j=e_j,~ j=1,\ldots,4.$
				\\\hline
				\caption{Left-symmetric product in $\mso\oplus\R e_4$ and $ \msl\oplus\R e_4$}
		\label{tab2}
		
\end{longtable}}
\begin{remark}
Two LSA's $\h^{\lambda_1> 0}_2$  and $\h^{\lambda_2> 0}_2$ are isomorphic if and only if $\lambda_1=\lambda_2$. They
are associative if and only if $\lambda_1 = 0$. In this case, $\h^{\lambda=0}_2$ coincides with the matrix algebra
$\mathcal{M}_2(\R)$.
\end{remark}

    \subsection{The new Lie brackets of Frobeniusian non-solvable Lie algebras}
{\renewcommand*{\arraystretch}{1.3}
	\begin{longtable}{|c| l| }\hline
		Algebra&Non-vanishing brackets\\\hline
		\multirow{5}{*}{ $\G_\rho^1$}&  $[f_1,f_2]=f_3, [f_1,f_3]=-2f_1,[f_2,f_3]=2f_2, [f_1,f_5]=-\frac{1}{2}f_6+f_7-f_8,$\\
		&$[f_1,f_7]=-\frac{1}{2}f_6, [f_1,f_8]=-\frac{1}{2}f_6,[f_2,f_5]=\frac{1}{2}(-f_5+f_8),[f_2,f_6]=-f_7-f_8,$\\
		&$[f_2,f_7]=\frac{1}{2}(f_5-f_8),[f_2,f_8]=\frac{1}{2}(-f_5+f_8),[f_3,f_5]=-f_5-f_7+f_8,[f_3,f_6]=f_6,$\\
		&$[f_3,f_7]=-f_8,[f_3,f_8]=-f_7,[f_4,f_5]=-f_5+\frac{1}{2}f_6+f_7+f_8,[f_4,f_6]=-f_6,$\\
		&$[f_4,f_7]= -\frac{1}{2} f_6-f_7,[f_5,f_8]=\frac{1}{2}f_6-f_8$.\\\hline
		\multirow{5}{*}{ $\G_\rho^2$}&
		
		$[f_1,f_2]=f_3,[f_1,f_3]=-2f_1,[f_2,f_3]=2f_2,[f_1,f_5]=f_7-(1+\lambda)f_8, $\\
		&$[f_1,f_7]=-\frac{1+\lambda}{2}f_6,[f_1,f_8]=-\frac{1}{2}f_6,[f_2,f_6]=-f_7-(1-\lambda)f_8,[f_2,f_7]=\frac{1-\lambda}{2}f_5,$\\
		&$[f_2,f_8]=-\frac{1}{2}f_5,[f_3,f_5]=-f_5,[f_3,f_6]=f_6,[f_3,f_7]=-\lambda f_7-(1-\lambda^2)f_8,$\\
		&$[f_3,f_8]=-f_7+\lambda f_8,[f_4,f_5]=-(1+\lambda)f_5,[f_4,f_6]=(\lambda-1)f_6,$\\
		&$[f_4,f_7]=(\lambda^2-1)f_7+\lambda(1-\lambda^2)f_8,[f_4,f_8]=\lambda f_7-(1+\lambda^2)f_8$.\\\hline
		\multirow{5}{*}{ $\G_\rho^3$}&
		$[f_1,f_2]=f_3,[f_1,f_3]=-2f_1,[f_2,f_3]=2f_2,[f_1,f_5]=f_7-f_8,[f_1,f_7]=-3f_5,$\\
		&$[f_1,f_8]=-3(f_5+f_6),[f_2,f_5]=-f_7,[f_2,f_6]=f_7-f_8,[f_2,f_7]=f_5+\frac{1}{4}f_6,$\\
		&$[f_2,f_8]=-3f_5-\frac{3}{4}f_6,[f_3,f_5]=-f_5-f_6,[f_3,f_6]=3f_6,[f_3,f_7]=-2f_7-f_8,$\\
		&$[f_3,f_8]=-3f_7,[f_4,f_j]=-f_j, j=5,6,7,8$.
		\\\hline
		\multirow{5}{*}{ $\G_\rho^4$}&
		$[f_1,f_2]=f_3,[f_1,f_3]=-2f_1,[f_2,f_3]=2f_2,[f_1,f_5]=-\frac{1}{2}f_6+f_7+(-1+\frac{\nu}{2})f_8,$\\
		&$[f_1,f_6]=-\frac{\nu}{2}f_6+\frac{\nu^2}{2}f_8,[f_1,f_7]=-\frac{1}{2}f_6+\frac{\nu}{2}f_8,[f_1,f_8]=-\frac{1}{2}f_6+\frac{\nu}{2}f_8,$\\
		&$[f_2,f_5]=\frac{1}{2}(-f_5+f_8),[f_2,f_6]=-\frac{\nu}{2}f_5-f_7+(-1+\frac{\nu}{2})f_8,[f_2,f_7]=\frac{1}{2}(f_5-f_8),$\\
		&$[f_2,f_8]=\frac{1}{2}(-f_5+f_8),[f_3,f_5]=-f_5-f_7+f_8,[f_3,f_6]=f_6-\nu(f_7+f_8),$\\
		&$[f_3,f_7]=-f_8,[f_3,f_8]=-f_7,[f_4,f_5]=(-1+\frac{\nu}{2})f_5+\frac{1}{2}f_6+f_7+(1-\nu)f_8,$\\
		&$[f_4,f_6]=\frac{\nu^2}{2}f_5+(-1+\frac{\nu}{2})f_6-\nu f_7+\nu(1-\nu)f_8,[f_4,f_7]=\frac{\nu}{2}f_5-\frac{1}{2}f_6-f_7,$\\
		&$[f_4,f_8]=\frac{\nu}{2}f_5+\frac{1}{2}f_6-(\nu+1)f_8$.
		\\\hline
		\multirow{5}{*}{ $\G_\rho^5$}&
		$[f_1,f_2]=f_3,[f_1,f_3]=-f_2,[f_2,f_3]=f_1,[f_1,f_5]=\frac{\mu}{4}f_5-(\frac{\mu^2}{4}+1)f_8,[f_1,f_6]=\frac{1}{2}f_7,$\\
		&$[f_1,f_7]=-\frac{1}{2}f_6,[f_2,f_5]=\frac{\mu}{4}f_6-\frac{1}{2}f_7,[f_2,f_6]=-f_8,[f_2,f_7]=\frac{1}{2}f_5-\frac{\mu}{2}f_8,$\\
		&$[f_2,f_8]=\frac{1}{4}f_6,[f_3,f_5]=\frac{1}{2}f_6+\frac{\mu}{4}f_7,[f_3,f_6]=-\frac{1}{2}f_5+\frac{\mu}{2}f_8,[f_3,f_7]=-f_8,$\\
		&$[f_3,f_8]=\frac{1}{4}f_7,[f_4,f_5]=-(\frac{\mu^2}{4}+1)f_5+\frac{\mu(\mu^2+4)}{4}f_8,[f_4,f_6]=-f_6+\frac{\mu}{2}f_7,$\\
		&$[f_4,f_7]=-\frac{\mu}{2}f_6-f_7,[f_4,f_8]=-\frac{\mu}{4}f_5+(-1+\frac{\mu^2}{4})f_8$.
		\\\hline
		\multirow{5}{*}{ $\G_\rho^6$}&
		$[f_1,f_2]=f_3,[f_1,f_3]=-f_2,[f_2,f_3]=f_1,[f_1,f_5]=-f_8,[f_1,f_6]=\frac{1}{2}f_7,$\\
		&$[f_1,f_7]=-\frac{1}{2}f_6,[f_1,f_8]=\frac{1}{4}f_5,[f_2,f_5]=-\frac{1}{2}f_7,[f_2,f_6]=-f_8,[f_2,f_7]=\frac{1}{2}f_5,$\\
		&$[f_2,f_8]=\frac{1}{4}f_6,[f_3,f_5]=\frac{1}{2}f_6,[f_3,f_6]=-\frac{1}{2}f_5,[f_3,f_7]=-f_8,[f_3,f_8]=\frac{1}{4}f_7,$\\
		&$[f_4,f_j]=-f_j, j=5,6,7,8.$\\\hline
		\caption{Lie algebras associated to $\G_{\rho}=\h\oplus\h^*$ }
		\label{tab3}
	\end{longtable}}

\subsection{Isomorphisms of $\G_\rho$ to Turkowski's algebras   }

	{\renewcommand*{\arraystretch}{1.3}
	\begin{longtable}{|c| l| l|}\hline
		Source&Isomorphism&Target\\\hline
		\multirow{2}{*}{$\G^1_\rho$} &$f_1=e_3,f_2=e_1,
		f_3=e_2,f_4=-e_6,f_5=e_7+e_8,$&\multirow{2}{*}{$\mL_{8,16}$}\\
		&$f_6=-e_7+e_8, f_7=-e_5+e_8,
		f_8=e_4.$&\\\hline
		\multirow{2}{*}{$\G^2_\rho$}&$f_1=e_3,f_2=e_1,
		f_3=e_2,f_4=-e_7+\frac{2}{p+1}e_8,f_5=-e_5,$&\multirow{2}{*}{$\mL_{8,17}^{p\in]-1,1]\setminus\{0\}}$}\\
		&$f_6=e_6,f_7=-e_7-\frac{2p}{p+1}e_8,f_8=\frac{p+1}{2}e_4,\lambda=-\frac{p-1}{p+1}.$&\\
		\hline
		\multirow{2}{*}{$\G^3_\rho$} &$f_1=e_3,f_2=e_1,f_3=e_2,f_4=-e_6,f_5=-e_7+e_8,$&\multirow{2}{*}{$\mL_{8,20}$}\\
		&$f_6=-2e_5-\frac{1}{2}e_6,f_7=\frac{1}{2}e_7+\frac{1}{6}e_8,f_8=e_4.$&
		
		\\\hline
		\multirow{2}{*}{ $\G^4_\rho$}&$f_1=e_3,f_2=e_1,f_3=e_2,f_4=-pe_6-\frac{1}{p}e_8,f_5=p(e_7+e_8),$&\multirow{2}{*}{$\mL_{8,18}^{p>0}$}\\
		&$f_6=-e_7+e_8,f_7=-e_5+e_8,f_8=pe_4, \nu=-\frac{1}{p^2}.$&    
		\\\hline
		\multirow{2}{*}{$\G^5_\rho$}&$f_1=e_1,f_2=e_2,f_3=e_3,
		f_4=e_5-e_7-\frac{2}{p}e_8,\mu=\frac{2}{p},$&\multirow{2}{*}{ $\mL_{8,4}^{p\neq0}$}\\
		&$f_5=-e_5-e_7+\frac{2}{p}e_8,f_6=e_6+2e_8,f_7=e_6-2e_8,
		f_8=pe_4.$&
		\\\hline
		\multirow{2}{*}{ $\G^6_\rho$} &$f_1=e_1,f_2=e_2,f_3=e_3,f_4=2e_8,
		f_5=e_6,f_6=e_7,$&\multirow{2}{*}{$\mL_{8,3}$}\\
		&$f_7=e_5,f_8=e_4.$&
		\\\hline
		\multirow{2}{*}{ $\G^7_\rho=\G^2_\rho(\lambda=1)$    } &$f_1=e_3,f_2=e_1,f_3=e_2,f_4=-e_7+e_8,f_5=-e_5,f_6=e_6,$&\multirow{2}{*}{$\mL_{8,17}^{p=0}(a_{67}=0)$}\\
		&$f_7=-e_7,f_8=\frac{1}{2}e_4.$&\\
		\hline
\caption{Isomorphisms from the Lie algebras obtained in Table \ref{tab3}  onto  Turkowski algebras \cite{T}.}
\label{tab4}
\end{longtable}}
\subsection{Proof of  Proposition $\ref{pr7}$ and $\ref{pr I}$ $($continued$)$}
Let $\G_{\rho,\al}^\ell=\h_\ell\oplus (\h_\ell)^\ast$ and
	$\{f_1,\ldots,f_8\}$ be a basis of $\G_{\rho,\al}^\ell$, with
	$\h_\ell=\langle  f_1,f_2,f_3,f_4\rangle$,
	$(\h_\ell)^*=\langle f_5,f_6,f_7,f_8\rangle$ and
	$\al(f_i,f_j)=\sum_{k=5}^{8}\al_{ij}^kf_k\in Z^2_\rho\big(\h_\ell,(\h_\ell)^\ast\big)\bigcap\mathcal{C}^2_L(\h_\ell,(\h_\ell)^\ast)$ for $i,j\in\{1,\ldots,4\}$ and $\ell\in\{1,3,4,5,6\}$. Let $\phi\in\mathcal{C}^1(\h_\ell,\h_\ell^\ast)$, such that $\partial_\rho\phi=\alpha$. In this case, $\phi$ is given by\\\\
$\ast$ \text{For the left-symmetric algebra $\h_1$.} 
\begin{eqnarray*}
\phi(f_1)&=&-\alpha_{14}^{8}f_5+\left({\frac {\alpha_{23}^{7}}{4}}+{\frac {3\,\alpha_{24}^{5}}{4}}-{
\frac {3\,{\alpha_{34}}^{7}}{8}}+{\frac {3\,\alpha_{34}^{6}}{4}}
\right)f_6+\left({\frac {\alpha_{34}^{8}}{2}}-{\frac {\alpha_{14}^{8}}{2}}-\alpha_
{13}^{6}+\alpha_{14}^{6}-{\frac {\alpha_{34}^{7}}{2}}
\right)f_7\\
&&+\left({\frac {3\,\alpha_{34}^{8}}{2}}-{\frac {\alpha_{14}^{8}}{2}}-\alpha_{13}^{6}+\alpha_{14}^{6}-{\frac {\alpha_{34}^{7}}{2}}
\right)f_8,\\
\phi(f_2)&=&\left({\frac {{\alpha_{23}}^{7}}{4}}+{\frac {3\,{\alpha_{24}}^{5}}{4}}-{
\frac {3\,{\alpha_{34}}^{7}}{8}}+{\frac {3\,{\alpha_{34}}^{6}}{4}}
\right)f_5+\left({\frac {\alpha_{23}^{7}}{4}}+{\frac {\alpha_{24}^{5}}{4}}-{\frac {
\alpha_{34}^{7}}{8}}+{\frac {\alpha_{34}^{6}}{4}}+\alpha_{24}^{6}\right)f_6\\
&&+\left(-{\frac {\alpha_{23}^{7}}{4}}+{\frac {\alpha_{24}^{5}}{4}}-{\frac 
{\alpha_{34}^{7}}{8}}+{\frac {\alpha_{34}^{6}}{4}}
\right)f_7,\\
\phi(f_3)&=&\left({\frac {\alpha_{34}^{8}}{2}}-{\frac {\alpha_{14}^{8}}{2}}-\alpha_{13}^{6}+\alpha_{14}^{6}-{\frac {\alpha_{34}^{7}}{2}}
\right)f_5+\left(-{\frac {\alpha_{23}^{7}}{4}}+{\frac {\alpha_{24}^{5}}{4}}-{\frac{\alpha_{34}^{7}}{8}}+{\frac {\alpha_{34}^{6}}{4}}
\right)f_6\\
&&+\left({\frac {\alpha_{23}^{7}}{2}}+{\frac {\alpha_{24}^{5}}{2}}-{\frac {
\alpha_{34}^{7}}{4}}+{\frac {3\,\alpha_{34}^{6}}{2}}-\alpha_{13}^{6}
\right)f_7+\left({\frac {\alpha_{14}^{8}}{2}}-{\frac {\alpha_{34}^{8}}{2}}+\alpha_
{24}^{5}-\alpha_{14}^{6}
\right)f_8,\\
\phi(f_4)&=&\left({\frac {3\,\alpha_{34}^{8}}{2}}-{\frac {\alpha_{14}^{8}}{2}}-
\alpha_{13}^{6}+\alpha_{14}^{6}-{\frac {\alpha_{34}^{7}}{2}}
\right)f_5+\left({\frac {\alpha_{14}^{8}}{2}}-{\frac {\alpha_{34}^{8}}{2}}+\alpha_
{24}^{5}-\alpha_{14}^{6}
\right)f_7\\
&&+\left({\frac {\alpha_{23}^{7}}{2}}+{\frac {\alpha_{24}^{5}}{2}}+\alpha_
{14}^{8}-2\,{\alpha_{34}}^{8}-2\,\alpha_{14}^{6}-{\frac {\alpha_{
34}^{7}}{4}}+{\frac {3\,\alpha_{34}^{6}}{2}}+\alpha_{13}^{6}
\right)f_8.
\end{eqnarray*}
$\ast$ \text{ For the left-symmetric algebra $\h_3$.}
\begin{eqnarray*}
\phi(f_1)&=&\left(-4\,\alpha_{14}^{6}-4\,\alpha_{24}^{6}+\alpha_{14}^{5}+4\, \alpha_{24}^{5}
\right)f_5+\left(-4\,\alpha_{24}^{6}+\alpha_{24}^{5}\right)f_6+\left({\frac {\alpha_{34}^{5}}{2}}-2\,\alpha_{34}^{6}-2\,\alpha_{12}^{
6}\right)f_7\\
&&+\left(-{\frac {\alpha_{34}^{5}}{2}}-2\,\alpha_{34}^{6}-2\,\alpha_{12}^{6}\right)f_8,\\
\phi(f_2)&=&\left(-4\,\alpha_{24}^{6}+\alpha_{24}^{5}\right)f_5+\left({\frac {3\,\alpha_{24}^{6}}{2}}-{\frac {5\,\alpha_{34}^{6}}{2}}-{
\frac {\alpha_{34}^{5}}{2}}-2\,\alpha_{12}^{6}\right)f_7+\left({\frac {\alpha_{34}^{6}}{2}}-{\frac {\alpha_{24}^{6}}{2}}\right)f_8,\\
\phi(f_3)&=&\left({\frac {\alpha_{34}^{5}}{2}}-2\,\alpha_{34}^{6}-2\,\alpha_{12}^{6}\right)f_5+\left({\frac {3\,\alpha_{24}^{6}}{2}}-{\frac {5\,\alpha_{34}^{6}}{2}}-{
\frac {\alpha_{34}^{5}}{2}}-2\,\alpha_{12}^{6}
\right)f_6\\
&&+\left(-3\,\alpha_{14}^{6}+4\,\alpha_{23}^{6}+16\,\alpha_{24}^{6}-
\alpha_{24}^{5}
\right)f_7+\left(-\alpha_{14}^{6}+\alpha_{24}^{5}\right)f_8,\\
\phi(f_4)&=&\left(-{\frac {\alpha_{34}^{5}}{2}}-2\,\alpha_{34}^{6}-2\,\alpha_{12}^{6}\right)f_5+\left({\frac {\alpha_{34}^{6}}{2}}-{\frac {\alpha_{24}^{6}}{2}}\right)f_6+\left(-\alpha_{14}^{6}+\alpha_{24}^{5}\right)f_7\\
&&+\left({\frac {\alpha_{24}^{5}}{3}}-{\frac {\alpha_{14}^{6}}{3}}-{\frac {
4\,\alpha_{24}^{6}}{3}}
\right)f_8.
\end{eqnarray*}
$\ast$ \text{For the left-symmetric algebra $\h^{\nu<0}_4$.}
\begin{eqnarray*}
\phi(f_1)&=&{\frac {\alpha_{13}^{5}}{3}}f_5+\left({\frac { \left( 2\,\nu-3 \right) \alpha_{23}^{6}}{6}}-{\frac {
\alpha_{12}^{6}}{2}}-{\frac {3\,\alpha_{24}^{6}}{2}}
\right)f_6+\left(-{\frac {\alpha_{13}^{6}\nu}{2}}+{\frac {\alpha_{13}^{5}}{6}}+{\frac {\alpha_{13}^{8}}{2}}+\alpha_{34}^{5}\right)f_7\\
&&+\left(-{\frac {\alpha_{23}^{6}{\nu}^{2}}{3}}+{\frac { \left( 3\,\alpha_{
12}^{6}+3\,\alpha_{13}^{6}+3\,\alpha_{23}^{6}+9\,\alpha_{24}^{6
} \right) \nu}{6}}+{\frac {\alpha_{13}^{5}}{6}}+{\frac {\alpha_{13}
^{8}}{2}}\right)f_8,\\
\phi(f_2)&=&\left({\frac { \left( 2\,\nu-3 \right) {\alpha_{23}}^{6}}{6}}-{\frac {\alpha_{12}^{6}}{2}}-{\frac {3\,\alpha_{24}^{6}}{2}}\right)f_5-{\frac {\alpha_{23}^{6}}{3}}f_6+\left(-{\frac {\alpha_{12}^{6}}{2}}+{\frac {\alpha_{23}^{6}}{6}}+{\frac 
{\alpha_{24}^{6}}{2}} \right)f_7,\\
\phi(f_3)&=&\left(-{\frac {\alpha_{13}^{6}\nu}{2}}+{\frac {\alpha_{13}^{5}}{6}}+{
\frac {\alpha_{13}^{8}}{2}}+\alpha_{34}^{5} \right)f_5+\left(-{\frac {\alpha_{12}^{6}}{2}}+{\frac {\alpha_{23}^{6}}{6}}+{\frac 
{\alpha_{24}^{6}}{2}} \right)f_6\\
&&+\left({\frac { \left( \nu-2 \right) \alpha_{23}^{6}}{3}}-\alpha_{13}^{6}
-2\,\alpha_{24}^{6}+\alpha_{34}^{7}
\right)f_7\\
&&+\left(-{\frac {\alpha_{13}^{8}}{2}}+{\frac { \left( 3\,\nu-6 \right)
\alpha_{13}^{6}}{6}}-{\frac {\alpha_{13}^{5}}{6}}+{\frac { \left( 3
\,\alpha_{12}^{6}+\alpha_{23}^{6}-3\,\alpha_{24}^{6} \right) \nu
}{6}}-\alpha_{34}^{7}-\alpha_{12}^{6}-{\frac {\alpha_{23}^{6}}{3
}}-\alpha_{24}^{6}-\alpha_{34}^{5}
\right)f_8,\\
\phi(f_4)&=&\left(-{\frac {\alpha_{23}^{6}{\nu}^{2}}{3}}+{\frac { \left( 3\,\alpha_{12}^{6}+3\,\alpha_{13}^{6}+3\,\alpha_{23}^{6}+9\,\alpha_{24}^{6} \right) \nu}{6}}+{\frac {\alpha_{13}^{5}}{6}}+{\frac {\alpha_{13}
^{8}}{2}} \right)f_5\\
&&+\left(-{\frac {\alpha_{13}^{8}}{2}}+{\frac { \left( 3\,\nu-6 \right)
\alpha_{13}^{6}}{6}}-{\frac {\alpha_{13}^{5}}{6}}+{\frac { \left( 3
\,\alpha_{12}^{6}+\alpha_{23}^{6}-3\,\alpha_{24}^{6} \right) \nu
}{6}}-\alpha_{34}^{7}-\alpha_{12}^{6}-{\frac {\alpha_{23}^{6}}{3
}}-\alpha_{24}^{6}-\alpha_{34}^{5}
\right)f_7\\
&&+\left({\frac {\alpha_{23}^{6}{\nu}^{2}}{3}}-\nu\,\alpha_{12}^{6}-2\,
\alpha_{13}^{6}\nu-\nu\,\alpha_{24}^{6}-{\frac {\alpha_{13}^{5}}{
3}}+\alpha_{13}^{6}-{\frac {2\,\alpha_{23}^{6}}{3}}-2\,\alpha_{14
}^{6}-2\,\alpha_{24}^{6}+2\,\alpha_{34}^{5}+\alpha_{34}^{7}
\right)f_8.
\end{eqnarray*}

$\ast$ \text{For the left-symmetric algebra $\h_5^{\mu>0}$.}
\begin{eqnarray*}
\phi(f_1)&=&\left({\frac {2\,\alpha_{13}^{6}}{3}}-{\frac {4\,\alpha_{23}^{5}}{3}}\right)f_5+\left({\frac {\mu\,\alpha_{13}^{8}}{24}}+{\frac { \left( {\mu}^{2}+12
 \right) \alpha_{13}^{5}}{24}}+{\frac {\mu\,\alpha_{23}^{7}}{6}}+{
\frac {\alpha_{14}^{6}}{4}}
\right)f_6\\
&&+\left(-{\frac {\alpha_{13}^{5}\mu}{6}}-{\frac {\alpha_{13}^{8}}{6}}-{
\frac {2\,\alpha_{23}^{7}}{3}}
\right)f_7+\left({\frac { \left( \alpha_{13}^{6}-2\,\alpha_{23}^{5} \right) \mu}{3}
}+2\,\alpha_{12}^{6}+3\,\alpha_{34}^{6}\right)f_8,\\
\phi(f_2)&=&\left({\frac {\mu\,\alpha_{13}^{8}}{24}}+{\frac { \left( {\mu}^{2}+12
 \right) \alpha_{13}^{5}}{24}}+{\frac {\mu\,\alpha_{23}^{7}}{6}}+{
\frac {\alpha_{14}^{6}}{4}}
\right)f_5+\left({\frac {4\,\alpha_{13}^{6}}{3}}-{\frac {2\,\alpha_{23}^{5}}{3}}\right)f_6\\
&&+\left({\frac { \left( -\alpha_{13}^{6}+2\,\alpha_{23}^{5} \right) \mu}{6
}}-\alpha_{12}^{6}-{\frac {\alpha_{34}^{6}}{2}}
\right)f_7\\
&&+\left(-{\frac {\alpha_{13}^{5}{\mu}^{3}}{24}}+{\frac { \left( -\alpha_{13}^{8}-4\,\alpha_{23}^{7} \right) {\mu}^{2}}{24}}+{\frac { \left( -6
\,\alpha_{14}^{6}+12\,\alpha_{13}^{5} \right) \mu}{24}}+\alpha_{
13}^{8} \right)f_8,\\
\phi(f_3)&=&\left(-{\frac {\alpha_{13}^{5}\mu}{6}}-{\frac {\alpha_{13}^{8}}{6}}-{
\frac {2\,\alpha_{23}^{7}}{3}}
\right)f_5+\left({\frac { \left( -\alpha_{13}^{6}+2\,\alpha_{23}^{5} \right) \mu}{6}}-\alpha_{12}^{6}-{\frac {\alpha_{34}^{6}}{2}}
\right)f_6\\
&&+\left(-{\frac {\alpha_{13}^{5}{\mu}^{2}}{12}}+{\frac { \left( -\alpha_{13}^{8}-4\,\alpha_{23}^{7} \right) \mu}{12}}+\alpha_{13}^{5}-{
\frac {3\,\alpha_{14}^{6}}{2}}
\right)f_8,\\
\phi(f_4)&=&\left({\frac { \left( \alpha_{13}^{6}-2\,\alpha_{23}^{5} \right) \mu}{3}}+2\,\alpha_{12}^{6}+3\,\alpha_{34}^{6}\right)f_5\\
&&+\left(-{\frac {\alpha_{13}^{5}{\mu}^{3}}{24}}+{\frac { \left( -\alpha_{13
}^{8}-4\,\alpha_{23}^{7} \right) {\mu}^{2}}{24}}+{\frac { \left( -6
\,\alpha_{14}^{6}+12\,\alpha_{13}^{5} \right) \mu}{24}}+\alpha_{
13}^{8}\right)f_6\\
&&+\left(-{\frac {\alpha_{13}^{5}{\mu}^{2}}{12}}+{\frac { \left( -\alpha_{13
}^{8}-4\,\alpha_{23}^{7} \right) \mu}{12}}+\alpha_{13}^{5}-{
\frac {3\,\alpha_{14}^{6}}{2}}
\right)f_7\\
&&+\left(-{\frac {4\,{\mu}^{2}\alpha_{13}^{6}}{3}}+{\frac {8\,{\mu}^{2}\alpha_{23}^{5}}{3}}-4\,\mu\,\alpha_{12}^{6}-6\,\mu\,\alpha_{34}^{6}-{\frac {8\,\alpha_{13}^{6}}{3}}+{\frac {16\,\alpha_{23}^{5}}{3
}}+4\,\alpha_{14}^{5}\right)f_8.
\end{eqnarray*}

$\ast$ \text{For the left-symmetric algebra $\h_6$.}
\begin{eqnarray*}
\phi(f_1)&=&\left(-2\,\alpha_{34}^{7}-2\,\alpha_{13}^{6}+2\,\alpha_{24}^{6}\right)f_5+\left(-\alpha_{23}^{6}-{\frac {\alpha_{14}^{6}}{2}}\right)f_6+\left(-\alpha_{23}^{7}-{\frac {\alpha_{14}^{7}}{2}}\right)f_7+\left(2\,\alpha_{13}^{7}-3\,\alpha_{24}^{7}\right)f_8,\\
\phi(f_2)&=&\left(-\alpha_{23}^{6}-{\frac {\alpha_{14}^{6}}{2}}\right)f_5+\left(-\alpha_{34}^{7}+\alpha_{24}^{6}\right)f_6+\left(\alpha_{13}^{7}-{\frac {\alpha_{24}^{7}}{2}}\right)f_7+\left(3\,\alpha_{14}^{7}+2\,\alpha_{23}^{7}\right)f_8,\\
\phi(f_3)&=&\left(-\alpha_{23}^{7}-{\frac {\alpha_{14}^{7}}{2}}\right)f_5+\left(\alpha_{13}^{7}-{\frac {\alpha_{24}^{7}}{2}}\right)f_6+\left(-3\,\alpha_{14}^{6}-2\,\alpha_{23}^{6}\right)f_8,\\
\phi(f_4)&=&\left(2\,\alpha_{13}^{7}-3\,\alpha_{24}^{7}\right)f_5+\left(3\,\alpha_{14}^{7}+2\,\alpha_{23}^{7}\right)f_6+\left(-3\,\alpha_{14}^{6}-2\,\alpha_{23}^{6}\right)f_7+4\,\alpha_{34}^{7}f_8.
\end{eqnarray*}
Consequently, we have $H^2_\rho(\h_\ell,\h_\ell^\ast)=0$. In addition, we can easily see that $\phi\in\mathcal{C}^1_L(\h_\ell,\h_\ell^\ast)$, which implies that $\partial_\rho\mathcal{C}^1_L(\h_\ell,\h_\ell^\ast)=C^2_L(\h_\ell,\h_\ell^\ast)\bigcap Z^2_\rho(\h_\ell,\h_\ell^\ast)$, i.e., $H^2_{L,\rho}(\h_\ell,\h_\ell^\ast)=0$ for all $\ell=1,3,4,5,6$.

\end{document}